\documentclass[11pt]{article}

\usepackage[T1]{fontenc}
\usepackage[utf8]{inputenc}
\usepackage{amsmath,amsfonts,amssymb,amsthm}
\usepackage{amsbsy}
\usepackage[a4paper, margin=2.2cm]{geometry}
\usepackage{stackengine}
\usepackage{mathtools}
\usepackage{graphicx}
\usepackage{dsfont}
\usepackage[normalem]{ulem}

\numberwithin{equation}{section}

\usepackage[backend=biber,maxnames=5,url=false,style=numeric-comp,isbn=false, giveninits=true,eprint=false,doi=false,date=year]{biblatex}
\renewbibmacro{in:}{}

\usepackage{color}
\definecolor{hanblue}{rgb}{0.27, 0.42, 0.81}
\definecolor{mordantred19}{rgb}{0.68, 0.05, 0.0}
\definecolor{darkgreen}{rgb}{0.0, 0.38, 0.12}
\definecolor{red}{rgb}{0.8, 0.0, 0.0}
\definecolor{green}{rgb}{0.0, 0.5, 0.0}

\usepackage[colorlinks, citecolor=hanblue,linkcolor=mordantred19,urlcolor=darkgreen]{hyperref}

\usepackage{xurl}
\usepackage[capitalise]{cleveref}

\usepackage{subcaption}
\usepackage{xfrac}
\usepackage{nicefrac}

\usepackage{makecell}

\usepackage[inline]{enumitem}

\allowdisplaybreaks

\newtheorem{theorem}{Theorem}[section]
\newtheorem{proposition}[theorem]{Proposition}
\newtheorem{lemma}[theorem]{Lemma}

\theoremstyle{remark}
\newtheorem{remark}[theorem]{Remark}
\theoremstyle{definition}
\newtheorem{definition}[theorem]{Definition}
\newtheorem{example}[theorem]{Example}

\numberwithin{equation}{section}

\newcommand{\iid}{\stackrel{\mathrm{iid}}{\sim}}
\newcommand{\N}{\mathbb{N}}
\newcommand{\Z}{\mathbb{Z}}
\newcommand{\R}{\mathbb{R}}

\newcommand{\norm}[1]{\left\Vert #1 \right\Vert}
\newcommand{\abs}[1]{\left\vert #1 \right\vert}

\DeclareMathOperator*{\argmin}{arg\,min}

\renewcommand{\d}{\,\mathrm{d}}

\newcommand{\calA}{\mathcal{A}}

\newcommand{\calP}{\mathcal{P}}

\newcommand{\eps}{\varepsilon}

\newcommand{\KR}{\mathrm{KR}}
\newcommand{\KRu}{\mathrm{KRu}}

\newcommand{\TV}{\mathrm{TV}}

\newcommand{\diam}{\operatorname{diam}}
\newcommand{\Lip}{\mathrm{Lip}}

\newcommand{\M}{\mathcal M}

\newcommand{\defeq}{:=}

\newcommand{\weakto}{\rightharpoonup}

\newcommand{\scal}[4]{{}_{#3}\langle #1,#2\rangle_{#4}}
\newcommand\restr[2]{{\left.\kern-\nulldelimiterspace #1 \vphantom{\big|} \right|_{#2}}}
\DeclareMathOperator*{\wstlim}{w^\ast\!-lim}
\newcommand{\mres}{\mathbin{\vrule height 1.4ex depth 0pt width
0.13ex\vrule height 0.13ex depth 0pt width 1.0ex}}

\let\sp\relax
\newcommand{\sp}[1]{\left\langle #1 \right\rangle}
\renewcommand{\phi}{\varphi}

\newcommand{\lsc}{lower-semicontinuous}

\DeclareMathOperator{\Ext}{Ext}

\newcommand\T{\rule{0pt}{2.6ex}}       % Top strut
\newcommand\B{\rule[-1.2ex]{0pt}{0pt}} % Bottom strut

\bibliography{references.bib}

\title{A Lipschitz spaces view of infinitely wide shallow neural networks}
\author{Francesca Bartolucci\thanks{ Institute of Applied Mathematics, TU Delft, Mekelweg 4, 2628 CD Delft, The Netherlands \\(\texttt{f.bartolucci@tudelft.nl})} \and Marcello Carioni\thanks{Department of Applied Mathematics, University of Twente, 7500AE Enschede, The Netherlands \\
(\texttt{m.carioni@utwente.nl}, \texttt{jose.iglesias@utwente.nl})} \and Jos\'e A. Iglesias\footnotemark[2] \and Yury Korolev\thanks{Department of Mathematical Sciences, University of Bath, Claverton Down, BA2 7AY, UK (\texttt{ymk30@bath.ac.uk})}  \and Emanuele Naldi\thanks{MaLGa - DIMA, University of Genoa, 16146 Genoa, Italy (\texttt{emanuele.naldi@edu.unige.it})} \and Stefano Vigogna\thanks{RoMaDS - Department of Mathematics,
University of Rome Tor Vergata
Rome, RM 00133, Italy\\ (\texttt{vigogna@mat.uniroma2.it})}}
\date{}

\begin{document}

\maketitle

\begin{abstract}
We revisit the mean field parametrization of shallow neural networks, using signed measures on unbounded parameter spaces and duality pairings that take into account the regularity and growth of activation functions. This setting directly leads to the use of unbalanced Kantorovich-Rubinstein norms defined by duality with Lipschitz functions, and of spaces of measures dual to those of continuous functions with controlled growth. These allow to make transparent the need for total variation and moment bounds or penalization to obtain existence of minimizers of variational formulations, under which we prove a compactness result in strong Kantorovich-Rubinstein norm, and in the absence of which we show several examples demonstrating undesirable behavior. Further, the Kantorovich-Rubinstein setting enables us to combine the advantages of a completely linear parametrization and ensuing reproducing kernel Banach space framework with optimal transport insights. We showcase this synergy with representer theorems and uniform large data limits for empirical risk minimization, and in proposed formulations for distillation and fusion applications.
\end{abstract}



\tableofcontents

\section{Introduction}

Understanding  theoretical properties of neural networks
is key to explaining their practical success.
In particular, functional analysis can shed light on their inductive bias,
identifying what classes of functions they can efficiently represent and learn.
Function spaces traditionally associated to learning models
are often characterized by smoothness conditions,
such as Sobolev regularity for kernel methods \cite{Wen05}.
However, neural networks seem to be driven by different priors,
notably forms of sparsity,
for which new functional structures may be required \cite{barron,bach2017breaking, e2019barron,Ongie2020A}.
Moreover, in seeming contradiction with classical models,
overparametrization is often beneficial in neural networks \cite{zhang,BelEtAl19},
which raises the question of how their complexity should be measured \cite{jiang2020fantastic}.
Defining functional norms provides a principled way to control
model complexity and enforce sparsity biases empirically observed.

% Neural networks are a mess, infinite width makes sense

At first glance, neural networks appear to be completely unstructured sets,
lacking linearity and a clear topological structure.
In addition to being nonlinear, the parametrization of a given architecture is arbitrarily fixed.
In practical scenarios, learning usually occurs in overparameterized regimes,
suggesting that networks operate in spaces with unbounded width.
Also from an approximation standpoint,
classical universality results indicate that the network size should be allowed to grow \cite{cybenko:1989, leshno:1993}.
Moreover, infinite width limits have been considered at initialization and through the training process,
where equivalence with Gaussian processes has been shown \cite{neal2012bayesian, lee2018deep, matthews2018gaussian}. \looseness=-1

% RKHS are not the answer

In general, it has been commonly accepted that
``the size of the weights is more important than the size of the network'' \cite{NIPS1996_fb2fcd53}.
But what is the right notion of size?
In the scaling limits that lead to Gaussian processes,
the network is completely characterized by a kernel,
be it the random feature or the neural tangent kernel \cite{rahimi,jacot2018neural}.
Thus, in these regimes, neural network are equivalent to or well approximated by kernel methods,
exhibiting lazy to no feature learning \cite{chizat2019lazy}.
The corresponding functional structure is given by reproducing kernel Hilbert spaces (RKHS) \cite{PauRag16},
whose norm is incapable of capturing the actual inductive bias of neural networks \cite{bietti2021deep, GhoEtAl21}.

% The trick of measures

Besides by limit, infinite width neural networks can be obtained by \emph{integration}.
The main idea is to rethink the network as parameterized on a space of measures,
rather than on the space of weights,
with finite atomic measures corresponding to finite width networks \cite{bach2017breaking}.
This approach opens to a series of good properties.
\begin{enumerate*}
 \item Linearity: the space of neural networks becomes linear (in the shallow case),
 or parameterized by a linear space (in the deep case).
 \item Topological structure: neural networks naturally inherit metrics and norms from measure spaces,
 such as the Wasserstein distance and the total variation norm.
 \item Nonparametric model: the space of neural networks contains now every possible finite architecture,
 which can therefore be problem- and data-dependent, rather than arbitrarily fixed.
 \item Feature learning: optimizing on the space of measures,
 interesting training dynamics can occur,
 for instance through Wasserstein gradient flow in the mean field regime \cite{mei2018,chizat2018global}.
\end{enumerate*}

% NN live in RKBS

Parametrizing neural networks over measures,
bias and training deviate from classical Hilbert regimes,
leading to more general Banach structures.
Yet, in this structural transition,
neural networks maintain an interesting functional similarity with kernel methods.
Analogously to the elements of a RKHS,
neural networks can be seen as hyperplanes in a feature space:
while such a feature space is Hilbertian for RKHS,
it is a Banach space (of measures) for neural networks.
The resulting structure is that of a reproducing kernel \emph{Banach} space (RKBS) \cite{zhang2009reproducing,linrkbs}.
Following this formalization, \cite{bartolucci2023understanding,bartolucci2024neural} have defined hypothesis classes and derived representer theorems for shallow and deep networks.
In functional terms, the complexity can be measured by a total variation norm,
characterizing the inductive bias as a bias towards sparsity in feature space.
In the particular yet fundamental case of the ReLU activation function,
such a bias can be described more explicitly using splines in the Radon transform domain~\cite{Unser2017769,parhi:2021,parhi2022}.
Moreover, for any activation,
the norm of finite parametric solutions (as provided by the representer theorem)
is controlled by an $\ell^1$ norm of the network coefficients.

% On duality and weights

The general construction of neural RKBS is based on a suitable interpretation
of the integration of neurons against a measure,
and precisely on the way we understand the integration as a pairing between dual spaces.
In \cite{bartolucci2023understanding}, the chosen duality is the one between Radon measures and continuous functions vanishing at infinity.
Since no common activation function has such a decay,
a suitably decaying weight function needs to be multiplied to make the integral converge.
As a result, the norm in the RKBS gets rescaled by this function,
and the norm of parametric solutions is actually a \emph{weighted} $\ell^1$ norm of the network coefficients.
The specific form of the inductive bias of RKBS neural networks therefore depends on the choice of this weight function.
While technically necessary and fairly easy to exemplify, 
the role of this weight, and hence the real nature of the inductive bias, is still not completely understood.

% In this paper

In this paper we aim to provide a deeper understanding of neural network functional bias,
along with an alternative description of neural RKBS.
Our technical starting point is rethinking the integration of neurons,
replacing the Riesz-Markov duality with that of Kantorovich-Rubinstein.
While typical activation functions do not vanish at infinity,
many are indeed Lipschitz continuous,
which makes this choice very natural.
Our formulation directly suggests the Kantorovich-Rubinstein norm as a measure of complexity and a natural form of regularization.
Moreover, the space of neural networks acquires a new interesting structure:
the distance between two networks can now be measured by a certain unbalanced Wasserstein metric,
drawing a connection to optimal transport.
This provides a variational framework for a variety of applications,
spanning from knowledge distillation \cite{hinton2015distilling, bucilua2006model} and model fusion \cite{singh2020model,akash2022wasserstein}
to linear mode connectivity \cite{pmlr-v238-ferbach24a}.
Taking this perspective also allows to extend the mean field analysis of neural networks
from a population to a sample setting,
studying strong convergence in Wasserstein distance of solutions to the empirical risk minimization (ERM) in the infinite data limit. Let us also emphasize that in contrast to \cite{SpeHeeSchBru22}, we are not studying dual pairings for the standard RKBS built from Riesz-Markov duality, but rather using a different underlying duality in the construction of the RKBS.

% Representer theorems and compactness

The definition of a nonparametric hypothesis space
makes ERM an infinite-dimensional problem.
Thus, a representer theorem, providing a workable parametric form of solutions, becomes desirable.
In the case of neural networks, it also serves as proof that
the proposed space of integral functions is a good mathematical model for the finite width architectures that are optimized in practice.
To ensure that any solutions exist through a standard calculus of variations argument, 
one needs first to establish the compactness of the unit ball with respect to the regularizing norm.
However, since the Kantorovich-Rubinstein space is not in general a dual space,
the Banach-Alaoglu theorem can not be applied.
To address this problem, we prove that
the total variation ball is compact in the Kantorovich-Rubinstein topology
under suitable moment conditions.
More precisely, we show that having finite moments of order strictly larger than $1$
is both necessary and sufficient to obtain compactness.
This result supports the adoption of a combined regularization consisting of three terms:
the Kantorovich-Rubinstein norm, the total variation norm, and a moment of sufficiently high order.
With this regularization, we prove a representer theorem for ERM
over the space of Kantorovich-Rubinstein infinitely wide neural networks,
establishing the existence of solutions of finite width.

% More on duality

The finite moment requirement suggests that the weight function introduced in \cite{bartolucci2023understanding}
should be regarded as an attribute of the parameterizing measure, rather than a way to artificially impose on the activation function a decay that it does not have.
With this insight, we further revise the neuron-measure duality,
interpreting the integral as a pairing between the space of weighted measures and its predual,
the space of continuous functions with controlled growth. 
On the one hand, this view extends the RKBS construction of neural networks beyond Lipschitz activations (such as rectified power units).
On the other hand, it jointly generalizes the Kantorovich-Rubinstein and Riesz-Markov neural network spaces.
Indeed, the Kantorovich-Rubinstein case can be recovered taking polynomial weights,
while the Riesz-Markov duality is obtained appending the measure weight to the activation.
This clarifies that weights are not merely technical conditions, but crucial ingredients
for defining the bias structure of neural networks and tuning the desired order of regularization,
with the Kantorovich-Rubinstein framework providing a canonical family of weights in terms of moments.
We also mention related work on weighted Barron spaces on bounded domains~\cite{DevNowParSie25} and on embeddings of Barron spaces arising from different activation functions~\cite{HerSpeSchBru24}.
% In summary

\medskip

\noindent \textbf{Outline of the main contributions.} In summary, our main contributions are:
\begin{enumerate}
 \item We introduce new hypothesis spaces for neural networks with Lipschitz activations based on the Kantorovich-Rubinstein duality.
 These spaces are RKBS, with a norm that naturally connects neural networks to optimal transport.
 \item Based on this connection, we provide natural variational formulations of teacher-student training and fusion of pre-trained networks.
 Moreover, we complement the mean field analysis of neural networks proving a result of strong convergence (in the Kantorovich-Rubinstein norm) of empirical risk minimizers in the infinite data limit, which in turn ensures uniform convergence on bounded sets for the network realizations.
 \item To show existence of solutions to the ERM problem,
 we prove a new compactness result in the Kantorovich-Rubinstein topology
 under the necessary and sufficient condition of finiteness of moments of order strictly greater than $1$.
 This theorem extends known results from compact to non-compact metric spaces and  is of independent interest.
 \item Using this result, we prove representer theorems for the Kantorovich-Rubinstein RKBS, showing the existence of ERM solutions of finite width.
 This justifies the proposed function spaces as nonparametric hypothesis classes for neural networks
 and provides a functional characterization of their inductive bias.
 \item We generalize the RKBS  structure using the duality between continuous functions with controlled growth and weighted measures.
 This allows us to clarify the role of weights in RKBS neural parameterizations
 and suggests polynomial weights as a canonical choice in the Kantorovich-Rubinstein case.
 \end{enumerate}
 
 % Organization of the paper
 
\noindent \textbf{Organization of the paper.} The remainder of the paper is organized as follows.
In Section~\ref{sec:two}, after recalling basic definitions of infinite-width shallow neural networks and RKBS, we define the Kantorovich-Rubinstein norm, we introduce the ERM problem and we present applications to the distillation and fusion of neural networks.
In Section~\ref{sec:comp}, we discuss the compactness of the Kantorovich-Rubinstein ball under moment bounds.
In Section~\ref{sec:regprobs}, we introduce the general optimization problem where the Kantorovich-Rubinstein norm is used as a regularizer, and we prove its well-posedness. We also propose several examples showing that the lack of such a moment condition leads to ill-posedness.
In Section~\ref{sec:controlledgrowth}, we introduce the space of continuous functions with controlled growth and we define the corresponding RKBS.
In Section~\ref{sec:representer}, we present several representer theorems for the general optimization problem including Kantorovich-Rubinstein norm penalization, concluding with applications to the distillation of neural networks.
In Section~\ref{sec:samplinglimit}, we study the convergence in Kantorovich-Rubinstein norm of empirical risk minimizers in the infinite data limit. 
In Section~\ref{sec:experiments}, we provide a simple numerical experiment on neural network distillation, exploiting the representer theorems of Section~\ref{sec:representer} in order to set up the corresponding finite-dimensional problem. We discuss how the student network constructed using the $\KR$ norm as penalty leads to reconstruction with different behavior than the student network obtained using unstructured sparsity promoted by the more usual total variation penalty.
For readers' convenience, we summarize our notation in \cref{tab:lip}. 

\begin{table}[t]

\caption{Lipschitz functions and measures on a generic space $Z$ (such as one of parameters $Z=\Theta$)} \label{tab:lip}

\centering

\resizebox{\columnwidth}{!}{

\begin{tabular}{l l l l}

\hline

$Z$ & \multicolumn{3}{l}{pointed metric space with metric $d$ and base point $e$}  \T\\

$C_b(Z)$ & \multicolumn{3}{l}{continuous bounded functions on $Z$, used for narrow convergence in $\M(Z)$} \T\\

$C_0(Z)$ & \multicolumn{3}{l}{continuous functions vanishing at infinity on $Z$} \T\\

$\Lip(Z)$ & Lipschitz functions on $Z$ & $\Lip_0(Z)$ & Lipschitz functions on $Z$ vanishing at $e$ \\

$ L(f) $ & Lipschitz constant of $f$ \\

$ \| f \|_{\Lip} $ & $ \max \{ |f(e)| , L(f) \} $ & $ \| f \|_{\Lip_0} $ & $ L(f) $ \B\\

\hline

\multicolumn{4}{c}{$ \Lip(Z) \cong \Lip_0(Z) \oplus \R $} \T\B\\

\hline

$\M(Z)$ & bounded measures on $Z$  & $\M^0(Z)$ & balanced measures on $Z$, i.e. $\mu(Z)=0$.\B\\

$ \|\mu\|_\TV $ & TV norm of $\mu \in \M(Z) = C_0(Z)^*$ \\

$\M_q(Z)$ & measures with finite $q$-moment & $\M_q^0(Z)$ & balanced measures in $\M_q(Z)$ \B \\

\hline

\multicolumn{4}{c}{$ \M_1(Z) \cong \M_1^0(Z) \oplus \R $} \T\B\\

\hline

$\| \mu \|_\KRu$ & $\| \mu \|_\KR + |\mu(Z)| $ & $\| \mu \|_\KR$ & Kantorovich-Rubinstein norm of $\mu$ \T\\

$ \KRu(Z) $ & completion of $\M_1(Z)$ under $\|\cdot\|_\KRu$ & $ \KR(Z) $ & completion of $\M_1^0(Z)$ under $\|\cdot\|_\KR$ \B \\

\hline

\multicolumn{4}{c}{$ \KRu(Z) \cong \KR(Z) \oplus \R $} \T\B \\

\hline

\multicolumn{2}{c}{$ \KRu(Z)^* \cong \Lip(Z) $} & \multicolumn{2}{c}{$ \KR(Z)^* \cong \Lip_0(Z) $} \T\B\\

\hline

\end{tabular}

}

\end{table}

\section{Reproducing Kernel Banach Spaces meet Kantorovich-Rubinstein norms}\label{sec:two}
\subsection{Infinite-width shallow neural networks}\label{sec:infwidth}
With the goal of fixing notation for the next sections we start by recalling the notion of shallow neural networks. Given a nonlinear activation function $\sigma : \R \rightarrow \R$, shallow neural networks with $N$ neurons are the class of parametrized functions defined as
\begin{align}\label{eq:finitenetwork}
f_{W,a,b}(x) = \frac{1}{N} \sum_{i=1}^N a_i \sigma(\langle w_i,x \rangle + b_i),\quad x \in \R^d,
\end{align}
with parameters $W \in \R^{N\times d}$, $a \in \R^N$ and $b \in \R^N$, where we denote by $w_i$ the $i$-th row of $W$.
A popular choice for the activation function $\sigma: \R \rightarrow \R$ is the so-called ReLU, defined as $\sigma(z) = \max\{z,0\}$. 

Shallow neural networks are parametrized nonlinearly. It has been noted (e.g., \cite{bach2017breaking}), however, that it is possible to overcome the obstacle of the nonlinearity by considering infinite-width shallow neural networks which are  parametrized linearly by Radon measures $\mu\in \mathcal{M}(\Theta)$ on the parameter space $\Theta = \{(w,b) \in \R^{d+1} \colon w \in \R^{d}, \, b \in \R\}$,
\begin{equation}\label{eq:meanfield}
    f_{\mu}(x) = \int_\Theta \sigma(\langle w, x\rangle+b)\, {\rm d}\mu(w,b),\qquad x\in\R^d.
\end{equation}
By choosing the empirical measure $\mu = \frac{1}{N}\sum_{i=1}^N a_i \delta_{(w_i, b_i)}$ in \eqref{eq:meanfield}, one  recovers a shallow neural network with $N$ neurons~\eqref{eq:finitenetwork}. Moreover, by identifying $x$ with $(x,1)$ one can consider $(w,b)$ as a single variable and write  
\begin{align}\label{eq:infinite_width_nn}
f_{\mu}(x) = \int_\Theta \sigma(\langle \theta, x\rangle)\, \d\mu(\theta), \qquad x\in \R^d \times \{1\} \subset \R^{d+1}.
\end{align}
For simplicity and given the recurrent role of the pairing $\langle \theta, x\rangle$ in the rest of the paper, in what follows we will always make this identification and, where a restriction may be warranted, denote the data space as $X \subseteq \R^{d} \times \{1\} \subset \R^{d+1}$.

As already pointed out, the space of infinite-width shallow neural networks $f_{\mu}$ is linear in $\mu \in \mathcal{M}(\Theta)$ and the training process becomes convex. 
Indeed, given a set of data and labels $(x_i,y_i)_{i=1}^N$, with $x_i \in \R^{d+1}$ and $y_i \in \R$, the learning problem for $f_\mu$ can be written as 
\begin{align}\label{eq:risk}
    \inf_{\mu \in \mathcal{M}(\Theta)} \frac{1}{N} \sum_{i=1}^N L(y_i, f_\mu(x_i)) + G(\mu),
\end{align}
where $L$ is a loss that is convex in the second component and $G$ is a convex regularizer with weak$^*$ compact sublevel sets.
Note that in the infinite width limit it is common to assume  compactness of the parameter space $\Theta \subset \R^{d+1}$. In particular, if one restricts the problem to measures on the $d$-dimensional sphere $\mathcal{M}(\mathbb{S}^d)$, the optimization problem \eqref{eq:risk} can be interpreted as an ERM problem for Barron functions \cite{weinan2022representation}. However, if the parameter space $\Theta$ is non-compact (which is often the case with commonly used neural networks), the integral \eqref{eq:infinite_width_nn} may not converge. In \cite{bartolucci2023understanding} the authors include a smoothing function in the infinite width limit of a shallow neural network to control the growth of the activation function and consequently ensure the convergence of the integral \eqref{eq:infinite_width_nn} for all measures. This approach allows one to extend to this non-compact setting the function spaces defined by neural networks, which turn out to be reproducing kernel Banach spaces parametrized by measures. 

\subsection{Function spaces of infinite-width shallow neural networks}
The study of infinite-width neural networks and their associated function spaces originates in the classical question of which functions can be approximated by finite-width neural networks \cite{hornik1989multilayer, cybenko:1989, barron, pinkus1999approximation}. Classical approximation theory focused on the relationship between approximation accuracy and network size, typically analyzing how fast the width must grow or how well a finite network can approximate a target up to a prescribed error. In this context, Barron introduced a class of functions, now known as Barron spaces, which capture those functions admitting dimension-free approximation by shallow neural networks \cite{barron}. Such functions can be expressed as infinite-width neural networks with bounded norm. A different functional-analytic perspective was introduced in \cite{savarese2019infinite}, where the authors studied approximation within the class of bounded-norm infinite-width networks. This perspective paved the way for a systematic investigation of  function spaces associated with infinite-width neural networks, along with the characterization of their norms. Notably, a norm characterization in terms of the Radon transform was presented in  \cite{Ongie2020A}, and later refined and extended in \cite{parhi:2021} and \cite{bartolucci2023understanding}. In recent years, the theory of function spaces associated with shallow neural networks has advanced significantly, producing a number of refined results \cite{bartolucci2023understanding, unser2023ridges, SpeHeeSchBru22, HerSpeSchBru24, DevNowParSie25}. In particular, the authors in \cite{bartolucci2023understanding} proposed a unified framework based on reproducing kernel Banach spaces, which encompasses and extends earlier constructions. It is worth stressing that mentioned results on infinite-width neural networks identify such function spaces through a total-variation?type regularization. In our work, for the first time, we move beyond this classical total variation framework and introduce a Kantorovich?Rubinstein regularization, resulting in a fundamentally different function space and associated norm. Finally, the study of analogous function spaces associated with deep neural networks remains in an exploratory phase \cite{parhi2022,bartolucci2024neural, heeringa2025deep, dummer2025vector} and constitutes an active and rapidly developing area of research.

\subsection{Reproducing kernel Banach spaces of shallow neural networks}

The definition of a reproducing kernel Hilbert space (RKHS) readily generalizes to the Banach space setting \cite{lin2022reproducing}. 

\begin{definition}[Reproducing kernel Banach space]
 Let $X$ be a set. A \emph{reproducing kernel Banach space} (RKBS) $\mathcal{B}$ over $ X $ is a Banach space of functions $f:X\to\R$ such that for all $x\in X$ there is a constant $ C_x > 0 $ such that
$
|f(x)|\leq C_x\|f\|_{\mathcal{B}}$ for all $ f \in \mathcal{B} $. In other words, for all $x\in X$ the map $f\mapsto f(x)$ is a continuous linear functional on $\mathcal{B}$.  
\end{definition}

The characterization of RKHSs in terms of feature maps is very popular in machine learning since it allows to see every function in a RKHS as a hyperplane in the codomain of the feature map, the so-called feature space. Interestingly, this characterization generalizes naturally to the Banach space setting.

\begin{proposition}[\cite{bartolucci2023understanding}]
 \label{prop:RKBS}
 A space $\mathcal{B}$ of functions on a set $ X$ is a RKBS
 if and only if
 there exist a Banach space $\mathcal{F}$ and a map $ \phi : X \to \mathcal{F}^* $ such that
 \begin{enumerate}
  \item[(i)] \label{it:rkbs-rep_prop}
  $ \mathcal{B} = \{ f_\mu : \mu \in \mathcal{F} \} $,\quad $ f_\mu (x) = {}_{\mathcal{F}}\langle \mu,\phi(x) \rangle_{\mathcal{F}^*} $;
  \item[(ii)] \label{it:rkbs-norm}
  $ \| f \|_\mathcal{B} = \inf \{ \| \mu \|_{\mathcal{F} }: \mu \in \mathcal{F} , f = f_\mu \} $.
\end{enumerate}
\end{proposition}
The map  $ \phi $ and the space $\mathcal{F}$ in Proposition~\ref{prop:RKBS} are respectively called \emph{feature map} and \emph{feature space}. Proposition~\ref{prop:RKBS} gives a recipe to construct a RKBS on a set $X$ starting from a Banach space $\mathcal{F}$ and a map $\phi\colon X\to\mathcal{F}^*$. In \cite{bartolucci2023understanding} the authors exploit Proposition~\ref{prop:RKBS} to define a class of RKBSs parametrized by the space of bounded measures. More precisely, they consider $\mathcal{F}=\mathcal{M}(\Theta)$, the Banach space of bounded measures on the parameter space $\Theta$, and define the feature map as 
$$
 \phi : X \to \mathcal{M}(\Theta)^*,
 \qquad \scal{\mu}{\phi(x)}{\mathcal M(\Theta)}{\mathcal M(\Theta)^*}=
 \int_\Theta \rho(x,\theta) \beta(\theta) {\rm d} \mu(\theta)  , 
 $$
 where $ \rho : X \times \Theta \to \R $ and  $ \beta : \Theta \to \R $ are $\theta$-measurable functions such that
\begin{equation}
\sup_{\theta\in\Theta} |\rho(x,\theta)  \beta(\theta)|=D_x< \infty \quad \forall \, x \in X ,
  \end{equation}
  for some $D_x>0$. Here, the function $\beta$ is needed to ensure that the
integral converges for all $\mu$ since $\rho$ may be unbounded. In particular, for $\rho\colon X\times\Theta\to\R$ of the form 
\[
\rho(x,\theta)=\sigma(\langle\theta,x\rangle),
\]
for some nonlinear activation  $\sigma\colon\R\to\R$,
the associated RKBS includes shallow neural networks with arbitrary width which correspond to measures having finite support.

Such function spaces are not uniquely defined by their compatibility with neural network parametrization. Equation \eqref{eq:infinite_width_nn} may be interpreted as a dual pairing between different pairs of suitable spaces and every interpretation gives rise to a different function space equipped with a corresponding norm. Studying the variety of function spaces defined by neural networks can help to understand the corresponding learning models and their inductive bias. 

\subsection{Lipschitz functions and unbalanced Kantorovich-Rubinstein norm on signed measures}
Since many  popular activation functions are Lipschitz, it seems reasonable to interpret the integral in \eqref{eq:infinite_width_nn} as the dual pairing 
\begin{align}\label{eq:pairingKR}
\scal{\mu}{\sigma(\langle\cdot,x\rangle)}{\KRu(\Theta)}{\Lip(\Theta)},
\end{align}
where $\Lip(\Theta)$ is the space of Lipschitz functions on $\Theta \subset \R^{d+1}$ endowed with a suitable norm and $\KRu(\Theta)$ will denote its predual, which contains the subspace $\mathcal{M}_1(\Theta)$ of measures with a finite first moment. In this Section~we introduce the necessary definitions and results regarding Lipschitz functions and Kantorovich-Rubinstein (KR) norms. For sake of generality we consider functions and measures defined in a general metric space $(Z,d)$. 
From now on, unless otherwise stated, we assume $(Z,d)$ to be a locally compact, complete, separable, geodesic, pointed metric space with base point $e \in Z$.
Let $\Lip(Z)$ denote the space of Lipschitz functions on $Z$  with the norm
\begin{equation}\label{eq:norm-Lip-e}
    \norm{f}_{\Lip} \defeq \max \left\{\abs{f(e)}, L(f) \right\}, \quad \text{where}\quad L(f) \defeq \sup_{z,z' \in Z,\ z\neq z'} \frac{\abs{f(z)-f(z')}}{d(z,z')}
\end{equation}
is the Lipschitz constant. The subspace of Lipschitz functions vanishing at the base point is denoted by $\Lip_0(Z)$ and its norm by $\norm{\cdot}_{\Lip_0} = L(\cdot)$.

Let denote by $\M(Z)$ the Banach space of bounded measures defined on the Borel $\sigma$-algebra of $Z$ endowed with the total variation (TV) norm $\|\cdot\|_{\TV}$. Since $Z$ is second countable, the elements of $\mathcal{M}(Z)$ are finite Radon measures, and the Markov-Riesz representation theorem ensures that
$\mathcal{M}(Z)$ can be identified with the dual of $C_0(Z)$, the Banach space of continuous functions going to zero at
infinity endowed with the sup norm $\|\cdot\|_{\infty}$. Then the TV norm is written as
\[
\|\mu\|_{\rm TV} = \sup\left\{\int_Z\psi(z){\rm d\mu}(z): \psi\in C_0(Z), \|\psi\|_{\infty}\leq1\right\}.
\]
We denote by $\M^0(Z)$ the subspace of balanced measures, i.e. those with $\mu(Z)=0$. Let $\mathcal{M}_p(Z)$, with $p\in(0,+\infty)$, be the subspace of measures in $\mathcal{M}(Z)$ with finite $p$-moments, i.e.
\begin{equation}\label{eq:def_of_Mp}
    \mathcal{M}_p(Z) \defeq \left\{\mu\in \mathcal{M}(Z):\int_Z d(z,e)^p \d\abs{\mu}(z)<\infty\right\}.
\end{equation}
For $p=1$, an alternative norm on $\M_1(Z)$ can be defined as follows~\cite[Sect. 8.10(viii)]{bogachev-measure-theory-vol2}
\begin{equation}\label{eq:KR-norm}
    \norm{\mu}_\KRu \defeq \abs{\mu(Z)} + \sup \left\{\int_Z f \, \d\mu(z) \colon f(e) = 0, \, L(f) \leq 1\right\}.
\end{equation}
For balanced measures the norm~\eqref{eq:KR-norm} is known as the Kantorovich-Rubinstein norm~\cite[Sec.~VIII.4]{kantorovich-akilov} denoted by $\|\cdot\|_{\KR}$. In our notation $\norm{\cdot}_\KRu$ we emphasize the fact that the measure can be unbalanced. The space $\M_1(Z)$ is not complete under this norm, and its completion will be denoted by $\KRu(Z)$. The completion of the subspace of balanced measures with a finite first moment (which we denote by $\M_1^0(Z)$) in the norm \eqref{eq:KR-norm} is the more standard Kantorovich-Rubinstein space $\KR(Z)$, also known as Arens-Eells \cite{AreEel56} or Lipschitz-free \cite{godefroy2003lipschitz} space over the pointed metric space $Z$. We first note that if $Z$ is compact, every sequence of measures uniformly bounded in variation and converging in $\KRu$ also converges weakly$^*$ in $\mathcal{M}(Z)$.

\begin{lemma}\label{lem:KR_implies_w*}
    Let $Z$ compact, $M>0$ and $\{\mu_n\}_{n\in \N}\subset \mathcal{M}(Z)$ be a sequence with $\|\mu_n\|_{\TV}\leq M$ for all $n\in\N$ and such that $\|\mu_n - \mu^*\|_{\KRu}\to 0$ as $n\to \infty$ for some $\mu^*\in \mathcal{M}(Z)$. Then $\{\mu_n\}_{n\in \N}$ converges weakly$^*$ to $\mu^*$.
\end{lemma}
\begin{proof}
    By the definition of the $\KRu$ norm, $\norm{\mu_{n} - \mu^*}_\KRu \to 0$ implies that 
\begin{align}\label{eq:narrow_convergence}
    \int_Z f(z)\, \d\mu_{n}(z)\rightarrow \int_Z f(z)\, \d\mu^*(z)
\end{align}
for every $f\in \Lip(Z)$ and then, by density, \eqref{eq:narrow_convergence} holds for every function in $C_0(Z)$, i.e. $\mu_{n}\rightarrow\mu^*$ weakly$^*$. Indeed, given $f\in C_0(Z)$, for every $\epsilon>0$ there exists $g_\epsilon\in\Lip(Z)$ such that $\|f-g_\epsilon\|_{\infty}<\epsilon$. Furthermore, there exists $N_\epsilon\in\N$ such that $| \int_Z g_\epsilon(z)\, \d\mu_{n}(z)- \int_Z g_\epsilon(z)\, \d\mu^*(z)|<\epsilon$ for all $n>N_\epsilon$. Then, for all $n>N_\epsilon$ 
\begin{align*}
  &\left| \int_Z f(z)\, \d\mu_{n}(z)- \int_Z f(z)\, \d\mu^*(z)\right|\\
  &\leq \|f-g_\epsilon\|_\infty\|\mu_n\|_\TV+\left| \int_Z g_\epsilon(z)\, \d\mu_{n}(z)- \int_Z g_\epsilon(z)\, \d\mu^*(z)\right|+\|f-g_\epsilon\|_\infty\|\mu^*\|_\TV\\
  &\leq\epsilon M+\epsilon+\epsilon\|\mu^*\|_\TV,
\end{align*}
which proves that \eqref{eq:narrow_convergence} holds for all $f\in C_0(Z)$.
\end{proof}

\begin{remark}\label{rem:kr-vs-wass}
    The convergence in $\KR$ norm alone is not sufficient to conclude weak$^*$ convergence. Take for example $Z= [-1,1]$ and $e=0$. Consider $\mu_r := \frac{1}{|r|^s}(\delta_r - \delta_{0})$ for $r\neq 0$, with $s \in (0,1)$. Clearly, $\mu_r \to 0$ as $r \to 0$ in the space $\KR$, but, on the other hand, considering the continuous (but not Lipschitz) function $f(z) = |z|^s$ we have $\int_{[-1,1]} |z|^s \, \d \mu_r(z) = 2$ for all $r\neq 0$. 
    %(EN: This was a note to myself. At first glance it looked weird to me, because on compact sets the topology induced by the Wasserstein-$1$ coincide with the weak$^*$ topology and $W_1(\mu,\nu)=\|\mu-\nu\|_\KR$ for all probability measures $\mu,\nu$. The point is that $\mu_z$ are balanced measures but not a difference between two probability measures.)
\end{remark}

We now state a well-known key result for our analysis and provide its proof for the sake of completeness. We refer also to 
\footnote{\url{https://regularize.wordpress.com/2016/07/05/the-completion-of-the-radon-measures-with-respect-to-the-kantorovich-rubinstein-norm/}} 
for an alternative proof.

\begin{theorem}\label{thm:dualityKR}
   It holds that
    \begin{equation}
        (\KRu(Z))^* \cong \Lip(Z).
    \end{equation}
\end{theorem}
\begin{proof}
    Every measure $\mu \in \M_1(Z)$ can be decomposed as
 $
        \mu = \mu_0 + \mu(Z)\delta_e,
   $
    where $\mu_0$ is balanced and $\delta_e$ is the Dirac delta at the base point $e$. This implies
    \begin{equation}\label{eq:KRu-directsum}
        \M_1(Z) \cong \M_1^0(Z) \oplus \R,
    \end{equation}
    where the norm on the direct sum is defined as $\norm{(\nu,t)}_{\M_1^0(Z) \oplus \R} \defeq \norm{\nu}_{\KR}  + |t|$. We observe that the isomorphism 
    \[
\M_1(Z)\ni\mu\mapsto \left( \mu-\mu(Z)\delta_e,\mu(Z)\right)\in \M_1^0(Z)\oplus\R
    \]
    maps any Cauchy sequence in $(\M_1(Z),\|\cdot\|_\KRu)$ in a pair of Cauchy sequences in $(\M_1^0(Z),\|\cdot\|_\KR)$ and $\R$. Vice versa, its inverse
    \[
\M_1^0(Z)\oplus\R\ni(\nu,t)\mapsto \nu+t\delta_e\in \M_1(Z)
    \]
    maps any pair of Cauchy sequences in $(\M^0_1(Z),\|\cdot\|_\KR)$ and $\R$ to a Cauchy sequence in $(\M_1(Z),\|\cdot\|_\KRu)$. As a consequence, the completion of $\M_1(Z)$ with respect to the KR-norm~\eqref{eq:KR-norm} is isometrically isomorphic to the direct sum of the completion of $\M_1^0(Z)$ with respect to the KR-norm and $\R$
    \begin{equation}\label{eq:KRudirectsum}
\KRu(Z)=\overline{\M_1(Z)}\cong \overline{\M_1^0(Z)}\oplus\R=\KR(Z)\oplus\R.
    \end{equation}
 Then, using the fact that $\KR(Z)^* \cong \Lip_0(Z)$~\cite[Thm. 3.3]{weaver:2018}, we derive
    \begin{equation*}
        (\KRu(Z))^* \cong \KR(Z)^* \oplus \R \cong \Lip_0(Z) \oplus \R,
    \end{equation*}
    where 
    the norm on $\Lip_0(Z) \oplus \R$ is defined as $\norm{(f,t)}_{\Lip_0(Z) \oplus \R} \defeq \max\left\{|t|,L(f) \right\}$. Clearly, $\Lip_0(Z) \oplus \R \cong \Lip(Z)$ through the isometric isomorphism 
    \[
    \Lip_0(Z) \oplus \R\ni(f,t)\mapsto f+t\in \Lip(Z),
    \]
which yields 
\[
(\KRu(Z))^* \cong  \Lip_0(Z) \oplus \R
\cong
\Lip(Z).
\]
\end{proof}

\begin{remark}\label{rem:unique_predual}
    In many situations, e.g. when $Z$ has a finite diameter or when $Z$ is a Banach space, $\Lip_0(Z)$ has a strongly unique predual~\cite[Sect. 3.4]{weaver:2018}.
\end{remark}

\subsection{\texorpdfstring{$\KRu$}{KRu}-norm based reproducing kernel Banach spaces}
%The assumption that $\Theta$ is compact turns out to be a limitation since the parameter space of commonly used neural networks is non-compact in general and the integral in \eqref{eq:infinite_width_nn} may not converge. However, we would still like to give a mathematical interpretation to formula \eqref{eq:infinite_width_nn}. 
Given a set of weights $\Theta \subset \R^{d+1}$ and choosing a base point $e \in \Theta$ (for a discussion on the choice of such base point we refer to Section~\ref{sec:opttrans}) we can interpret \eqref{eq:infinite_width_nn} as the dual pairing \eqref{eq:pairingKR} and introduce the space of functions 
\[
\mathcal{B}^{\KR}_\sigma=\{f_\mu:\mu\in\KRu(\Theta)\},\qquad f_\mu(x)=\scal{\mu}{\sigma(\langle\cdot,x\rangle)}{\KRu(\Theta)}{\Lip(\Theta)}.
\]
By Proposition~\ref{prop:RKBS}, $\mathcal{B}^{\KR}_\sigma$ is a reproducing kernel Banach space endowed with the norm
\[
\|f\|_{\mathcal{B}^{\KR}_\sigma}=\inf\{\|\mu\|_{\KRu}:f=f_\mu\}.
\]
Since the space $\mathcal{B}^{\KR}_\sigma$ is parametrized by the Kantorovich-Rubinstein space $\KRu(\Theta)$, the learning problem
\begin{align}\label{eq:risk_RKBS}
    \inf_{f \in \mathcal{B}^{\KR}_\sigma} \frac{1}{N} \sum_{i=1}^N L(y_i, f(x_i)) + \|f\|_{\mathcal{B}^{\KR}_\sigma}
\end{align}
can be reformulated as a minimization problem over $\KRu(\Theta)$
\begin{align}\label{eq:risk_KR}
    \inf_{\mu \in \KRu(\Theta)} \frac{1}{N} \sum_{i=1}^N L(y_i, f_\mu(x_i)) + \|\mu\|_{\KRu}.
\end{align}
In fact, if $\mu^*$ is a minimizer of \eqref{eq:risk_KR}, then $f_{\mu^*}$ is a minimizer of \eqref{eq:risk_RKBS}.
There are two problems with this formulation. First, elements of $\KRu(\Theta)$ may be too irregular to be useful in the context of neural networks; indeed, $\KRu(\Theta)$ contains elements which are not measures if $\Theta$ is not finite\footnote{\url{https://regularize.wordpress.com/2016/07/05/the-completion-of-the-radon-measures-with-respect-to-the-kantorovich-rubinstein-norm-i/}} (see also \cite[Theorem~3.19]{weaver:2018} considering \cref{rem:unique_predual}), thus losing the integral form \eqref{eq:infinite_width_nn}. An example is given below. 

\begin{example}Consider the sequence $\{\mu_n\}_{n\in\N}$ defined by
\begin{align}
    \mu_n = \sum_{i=1}^n (\delta_{\theta_i}-\delta_{0}) \in \mathcal{M}(\R), 
\end{align}
with $\theta_{i+1} < \frac{1}{2}\theta_i$ for every $i\in\N$ with $0 < \theta_i \leq 2$. Let $\mu^*$ be the limit of $\{\mu_n\}_{n\in\N}$ in $\KRu(\R)$. Note that such limit exists since $\{\mu_n\}_{n\in\N}$ is Cauchy in $\KRu(\R)$. Indeed, for $n > m$ and every $\varphi \in {\rm Lip}_0(\R)$ such that $L(\varphi) \leq 1$, it holds
\begin{align}
    \left|\int_{\R}\varphi(\theta) \d(\mu_n - \mu_m)(\theta)\right|  \leq \sum_{i=m}^n |\varphi(\theta_i)-\varphi(0)| \leq \sum_{i=m}^n \theta_i \leq \sum_{i=m}^{\infty} \theta_i \rightarrow 0
\end{align}
as $m \rightarrow +\infty$ since it is the tail of a converging series, due to the condition $\theta_{i+1} < \frac{1}{2}\theta_i$.
It remains to show that $\mu^*$ is not an element of $\mathcal{M}(\R)$. Let $\{\varphi_k\}_{k\in\N}$ be a sequence of continuous functions defined as $\varphi_k(\theta) = 0$ if $\theta \leq 0$, $\varphi_k(\theta) = k\theta$ if $0\leq \theta \leq 1/k$, $\varphi_k(\theta) = 1$ for $1/k \leq \theta \leq 2$ and extended continuously for $\theta \geq 2$ so that $\|\varphi_k\|_\infty = 1$ and $\lim_{\theta\rightarrow +\infty} \varphi_k(\theta) = 0$. Then, for every $k\in\N$, we have
\begin{align*}
    \scal{\mu^*}{\varphi_k}{\KRu(\R)}{\Lip(\R)} &= \lim_n  \sum_{i=1}^n (\varphi_k(\theta_i)-\varphi_k(0))  = \hat i(k) + k\sum_{i=\hat i(k)}^\infty \theta_i 
\end{align*}
where $\hat i(k) = \min\{i : \theta_i \leq \frac{1}{k}\}$. Therefore, since $\scal{\mu^*}{\varphi_k}{\KRu(\R)}{\Lip(\R)} \rightarrow +\infty$ as $k \rightarrow +\infty$ and $\|\varphi_k\|_\infty = 1$ for every $k$ we conclude that $\mu^*$ is not a bounded measure.  
\end{example}

Moreover, the $\KRu$ topology can behave in undesirable ways in the absence of a bound in the $\TV$ norm. For example, rescaled dipoles $\frac{1}{d(\theta,\theta_0)^s}(\delta_\theta - \delta_{\theta_0})$ with $s < 1$ converge to zero in $\KRu(\Theta)$ as $\theta \to\theta_0$, but in the context of neural networks as in \eqref{eq:finitenetwork} and \eqref{eq:meanfield} they correspond to two neurons having increasingly similar parameters $(w,b)$ but with opposite signs, and with output multiplied by ever larger coefficients $a$.\looseness=-1
   
The second problem is technical. The space $\KRu(\Theta)$ is not a dual space, in general, hence we cannot apply the Banach-Alaoglu theorem to guarantee (subsequential) convergence of minimizing sequences in the weak$^*$ topology. For example, the case of $\Theta = \R^{d+1}$ with the standard metric does not fall into the ones where $\KRu(\Theta)$ has a known predual. There are cases when $\KRu(\Theta)$ indeed has a predual (the so-called little Lipschitz space, see~\cite[Ch. 4]{weaver:2018}), for example when $\Theta = \R^{d+1}$ with the metric $d(\theta,\theta') \defeq \norm{\theta-\theta'}^s$ for $0 < s < 1$, but this excludes interesting activation functions such as ReLU (one can easily check that it is not Lipschitz with respect to this metric). To showcase the bad behavior of sequences in $\KRu(\Theta)$ in absence of a bound on the TV norm, we present an example where a sequences of dipoles does not have any subsequence converging weakly in $\KRu(\R)$.

\begin{example}\label{ex:dipolesnotconv}
We consider again a sequence of dipoles but with $s = 1$, that is, $\frac{1}{d(\theta,\theta_0)}(\delta_\theta - \delta_{\theta_0})$ with $\Theta = \R$, $e = \theta_0 = 0$, and $\theta_n = 1/n$. In particular, we consider the sequence $\{\mu_n\}_{n\in\N}$ in $\KR(\R)$ given as $\mu_n = n (\delta_{1/n}-\delta_0)$, for which clearly $\|\mu_n\|_{\KR}=1$. Given any subsequence $\{\mu_{n_k}\}_{k\in\N}$, we can take a further not relabelled subsequence to assume that $n_{k+1} \geq 3 n_k$, and based on it define the function 
\[f(t):=(-1)^k \left( \frac{1}{n_k} -\left(|t|- \frac{1}{n_k}\right)\frac{n_{k+1} + n_k}{n_k - n_{k+1}}\right) \quad \text{if }|t| \in \left[\frac{1}{n_{k+1}}, \frac{1}{n_k}\right),\]
extended continuously by $0$ at $t = 0$, for which by the condition $n_{k+1} \geq 3 n_k$ we get 
\[L(f) = \abs{\frac{n_{k+1} + n_k}{n_k-n_{k+1}}} = \frac{n_{k+1} + n_k}{n_{k+1}-n_k} \leq 2,\]
so that $f \in \Lip_0(\R)$. Using it, we have
\[\int_{\R} f(t) \d\mu_{n_k}(t) = n_k \left( f\big(1/n_k\big)-f(0)\right) = (-1)^k,\]
so the original subsequence, which was arbitrary, does not converge weakly in $\KRu(\R)$.
\end{example}

Despite all such negative results, it is known that for a compact space $\Theta$, the unit ball with respect to the total variation norm is strongly compact in the Kantorovich-Rubinstein space~\cite[Thm. VIII.4.3]{kantorovich-akilov}. Moreover, we will show that for a non-compact $\Theta$ the result holds under certain moment conditions on the measures, which are also necessary to ensure such compactness. 
For this reason we modify the regularizer in~\eqref{eq:risk_KR} by penalizing the total variation and $p$-moment of the measures. The previous observations lead to the following ERM problem.

\begin{definition}[$\KRu$-regularized empirical risk minimization]\label{def:KRrisk}
We consider for $p>0$ the following minimization problem:    
\begin{align}\label{eq:risk_new}
    \inf_{\mu \in \KRu(\Theta)} \frac{1}{N} \sum_{i=1}^N L(y_i, f_\mu(x_i))  +  G_{\alpha,\beta}(\mu) %+ \|\mu\|_{\M(\Theta)}
\end{align}
where $\alpha \geq 0$ and $\beta >0$ are parameters
%,$e \in \Theta$ is a base point, $d$ is a chosen metric on $\Theta$ 
and 
the regularizer $G_{\alpha,\beta} : \KRu(\Theta) \rightarrow [0,\infty]$ is defined by 
\begin{equation}\label{eq:G-definition-Rd}
    G_{\alpha,\beta}(\mu) \defeq
    \begin{cases}
        \alpha\|\mu\|_{\KRu} + \beta \displaystyle \int_\Theta ( 1 + d(\theta,e)^p ) \, \rm d |\mu|(\theta)   &  \text{if } \mu \in \mathcal{M}_p(\Theta), \\
         +\infty & \text{otherwise.} 
    \end{cases}
\end{equation}
\end{definition}
Note that if $\mu \in \mathcal{M}_p(\Theta)$,  then $G_{\alpha,\beta}(\mu)$ can be equivalently written as
\begin{align}
   G_{\alpha,\beta}(\mu) =  \alpha\|\mu\|_{\KRu} + \beta \|\mu\|_{\TV} + \beta \displaystyle \int_\Theta d(\theta,e)^p  \, \d\abs{\mu}(\theta).
\end{align}

The goal of the next sections is to analyze the variational problem \eqref{eq:risk_new}. 

\begin{remark}[The need of $p>1$]
We remark that the choice $p>1$ in the definition of $G_{\alpha,\beta}$ will be crucial for the well-posedness of the variational problem \eqref{eq:risk_new}.
Indeed, we will show that if $p>1$ the regularizer $G_{\alpha,\beta}$ enforces compactness on \eqref{eq:risk_new}, which yields existence of minimizers of \eqref{eq:risk_new} by the direct method of calculus of variations. On the other hand, if $p=1$, we lose any a priori compactness for the variational problem \eqref{eq:risk_new} as highlighted in the examples in Section~\ref{sec:nonexistence}. More crucially, in many cases of interest, this lack of compactness leads to non-existence of minimizers of \eqref{eq:risk_new}, making the ERM problem \eqref{eq:risk_new}  ill-posed, see Section~\ref{sec:nonexistence}. 
\end{remark}

\subsubsection{Optimal transport interpretation of the \texorpdfstring{$\KRu$}{KRu} norm and choice of base point}\label{sec:opttrans}
The penalization of the ${\rm KR}$-norm in \eqref{eq:risk_new} can be interpreted through the lens of optimal transport. Indeed, if $\mu \in \M^0(Z)$ is a balanced measure, i.e. such that $\mu(Z) = 0$, then 
\begin{align}\label{eq:otformula}
    \|\mu\|_{{\rm KRu}} = \|\mu\|_{{\rm KR}} =   W_1(\mu_+,\mu_-)
\end{align}
where $W_1$ is the 1-Wasserstein distance and $\mu_+$ and $\mu_-$ are the positive and the negative parts of $\mu$, respectively. In particular, the KR-norm in \eqref{eq:risk_new}  penalizes the optimal transport cost between the positive and the negative parts of $\mu$. If $\mu(Z) \neq 0$, then the characterization \eqref{eq:KRu-directsum} implies $\|\mu\|_{\KRu} = |\mu(Z)|+ \|\mu - \mu(Z)\delta_e\|_{\KR},$ which means that the excess of mass in the positive or the negative part of $\mu$, depending on the sign of $\mu(Z)$, is necessarily transported to the  base point $e$. 

This is a simple way to formulate unbalanced optimal transport, but by far not the only one. A more involved alternative in the signed measure setting is to use an infimal convolution between the 1-Wasserstein distance $W_1(\mu_+,\mu_-)$ and the total variation norm $\|\mu\|_{\TV}$, as first introduced in \cite{hanin1992} and whose sparsity aspects were studied in \cite{CarIglWal24p}. That type of formulation instead aims to make an optimal decision between transporting mass between $\mu^-$ and $\mu^+$, or creating it in one of them.

Naturally, for any particular application the base point $e$ needs to be chosen, in particular in the case where $\Theta$ is a parameter space for the mean field representation \eqref{eq:meanfield} with $\Theta = \R^d \times \R$ generalizing the discrete case in \eqref{eq:finitenetwork}, and the related regularized problems of the form \eqref{eq:risk_new}. We argue that in this case the canonical choice of base point is $e=(0,0) \in \R^d \times \R$, since the term $\mu(\Theta)\delta_{(0,0)}$ only affects the representation $f_\mu$ by adding a constant, which vanishes if $\sigma(0)=0$.

\subsubsection{Interpretation of the \texorpdfstring{$G_{\alpha,\beta}$}{Ga,b} regularization and comparison with commonly used regularizers}
Analyzing the influence of the regularizer $G_{\alpha,\beta}$ on the parameter space yields an interpretation of both the parameter configurations it promotes and its impact on the associated function $f_{\mu}$.
Below we summarize the main effects and the resulting bias and compare $G_{\alpha,\beta}$ with regularizers used in practice.
\paragraph{Function-space view}
On bounded input sets, say $ X \subset B(0,R) \times \{1\} $, and setting $ e = (0,0) $,
the mappings $ \theta \mapsto \sigma( \langle \theta , x \rangle  ) $ are uniformly Lipschitz
with $\Lip$ norm (defined as in \eqref{eq:norm-Lip-e}) bounded by $ \max\{ | \sigma(0) | , L(\sigma) \sqrt{R^2 + 1} \} $.
Hence, by the Kantorovich--Rubinstein duality, we obtain $ \| f_\mu \|_\infty \lesssim \|\mu\|_{\KRu} $.
Thus, the $\KRu$ norm directly controls the uniform amplitude of the solution.
Moreover, since for $p\geq 1$ it holds
\[\begin{aligned}
 | f_\mu(x) - f_\mu(x') | & \leq L(\sigma) \| x - x' \| \int_{\Theta} \| \theta \| \d|\mu| (\theta) \\
 & \leq L(\sigma) \| x - x' \| \int_{\Theta} (1+ \| \theta \|^p) \d|\mu| (\theta),
\end{aligned}\]
the total variation of $\mu$ together with its $p$-moment (with $p\geq 1$) bounds the Lipschitz constant of $f_\mu$.
Finer functional properties can be induced in specific cases.
For instance, in ReLU networks the $\TV$ norm controls the total mass of the kinks, and together with the $p$-moment they control their steepness, and the $\KRu$ norm governs their spatial arrangement, providing a measure of geometric complexity.
For general Lipschitz activations, a similar behavior occurs, with kinks replaced by localized features, e.g. curvature bumps.
\paragraph{Parameter-space view} In \Cref{sec:opttrans}, we recalled the connection between the $\KRu$ norm and optimal transport, which helps clarify what this norm promotes. If the base point $e$ is such that $\sigma(\langle e, x\rangle)=0$ for all $x$ (e.g.\ $e=(0,0)$ for ReLU, or any activation with $\sigma(0)=0$), then the $\KRu$ norm admits an interpretation as a cost of creating mass, together with a penalty for positive and negative contributions that are far apart in parameter space. On the other hand, the $\TV$ norm promotes sparsity, encouraging representations given by few atoms. While the total variation controls the total mass of the measure,
the $p$-moment controls its spread,
discouraging neurons from populating widely dispersed regions of parameter space.
\paragraph{Inductive bias}
Overall, the combination of all terms in $G_{\alpha,\beta}$ produces a function class with tightly controlled amplitude, slope, and geometric complexity,
with parameter configurations biased towards sparsity, concentration, and transport-optimality.
\paragraph{Relation to standard regularization techniques}
The regularizer $G_{\alpha,\beta}$ combines, within a single convex functional, several effects that in practical deep learning are usually implemented separately.
The total variation term serves as an infinite-width analogue of an $\ell^1$ penalty, akin to Barron norms and $\ell^1$-path norm in finite networks.
The $p$-moment controls the effective radius of the parameter distribution and bounds the Lipschitz constant of the realized function, thereby reproducing the regularizing influence of weight decay, spectral normalization, and gradient-norm penalties.
The unbalanced Kantorovich--Rubinstein term acts as a signed optimal-transport cost on parameter space, penalizing geometric displacement of mass and controlling the amplitude of the function.
This mechanism parallels Wasserstein penalties and Lipschitz constraints used in adversarial training and generative modeling \cite{arjovsky2017wasserstein, tolstikhin2017wasserstein, mukherjee2021end, patrini2020sinkhorn}.
In contrast to these practical techniques, $G_{\alpha,\beta}$ offers a unified framework that simultaneously enforces sparsity, smoothness, amplitude control, and geometric compactness of the infinite-width representation.

\subsubsection{Machine learning applications}\label{sec:applications}

Beyond the basic ERM problem \eqref{eq:risk_new}, using the $\KRu$ norm in the objective lends itself to suitable applications for the so-called distillation of infinitely wide neural networks, and their fusion.

\medskip 

\textbf{Distillation of an infinitely wide neural networks.} In classical neural networks literature, the goal of distillation is to transfer the ability of a given pre-trained network called \emph{teacher} to a smaller network called \emph{student} \cite{hinton2015distilling, bucilua2006model}. Several approaches to distillation and knowledge transfer have been later proposed exploiting optimal transport techniques \cite{li2020representation}. The popularity of optimal transport for distillation and compression can be justified by thinking that the student network should have weights that are close to the teacher network and therefore, in a measure theoretical sense, optimal transport metrics are the natural ones to achieve such an outcome.  
Distillation can be achieved naturally by generalizing our model \eqref{eq:risk_new} as follows 
\begin{align}\label{eq:risk_distillation}
    \inf_{\mu \in \KRu} \frac{1}{N} \sum_{i=1}^N L(y_i, f_\mu(x_i)) + \alpha\|\mu - \mu^*\|_{\KRu}+\beta\int_\Theta (1+ d^p(\theta,e)) \, \d\abs{\mu - \mu^*},
\end{align}
where $\mu^*$ is the teacher network. The penalization $\int_\Theta (1+d^p(\theta,e)) \, \d\abs{\mu - \mu^*}(\theta)$ is promoting further sparsity in the difference between the solution and the teacher network, therefore reducing the complexity of the distilled network. The term $\|\mu - \mu^*\|_{\KRu(\Theta)}$ penalizes transport between the weights of the teacher and the student networks and, in view of the duality of the previous section, enforces similarity of the resulting representations. Further training data could be used to refine the student network in the loss, resulting in the additional empirical risk loss $\frac{1}{N} \sum_{i=1}^N L(y_i, f_\mu(x_i))$.

\medskip

\textbf{Fusion of infinitely wide neural networks.}
Combining machine learning models in such a way that the combined network model is able to retain the ability of each single component is a crucial task in learning. Optimal transport approaches have been successfully proposed to achieve it \cite{singh2020model, imfeld2024transformer, akash2022wasserstein}. 
Fusion of pre-trained infinite wide neural networks can be achieved naturally by our model \eqref{eq:risk_new} by considering 
\begin{align}\label{eq:risk_fusion}
    \inf_{\mu \in \KRu(\Theta)} \frac{1}{N} \sum_{i=1}^N L(y_i, f_\mu(x_i)) + \alpha \sum_{i=1}^K \|\mu - \mu_i^*\|_{\KRu} +\beta \int_\Theta (1+ d^p(\theta,e)) \, \d\abs{\mu},
\end{align}
where $(\mu_i^*)_{i=1,\ldots,K}$ are the pre-trained neural networks that we would like to merge.
In particular, it is worth noting that \eqref{eq:risk_fusion} can be interpreted as the computation of a perturbed $\KR$-barycenter for the pre-trained neural networks $\mu_i^*$.

%%%%%%%%%%%%%%%%%%%%%%%%%%%%

\subsection{The case \texorpdfstring{$p=1$}{p=1} and \texorpdfstring{$\Lip^*$}{Lip*}}

The reader could wonder why, instead of the dual pairing from the previous section, one should not consider the duality pairing
\begin{align}\label{eq:pairinglip*}
\scal{\sigma(\langle\cdot,x\rangle)}{\mu}{\Lip(\Theta)}{\Lip^*(\Theta)},
\end{align}
where $\mu \in \mathcal{M}_1(\Theta)$ is embedded in $\Lip(\Theta)^*$ through the mapping $\mu\mapsto I_\mu$ defined as $I_\mu(f)=\int_\Theta f(\theta)\text{d}\mu(\theta)$,
hoping to apply Banach-Alaoglu theorem to infer existence of minimizers. In particular, interpreting \eqref{eq:infinite_width_nn} as the dual pairing \eqref{eq:pairinglip*}, 
one can set up an ERM problem in  $\mathcal{M}_1(\Theta)$ regularized with \begin{equation*}
    G(\mu) =  \int_\Theta ( 1 + d(\theta,e)^p ) \, \d\abs{\mu}(\theta)
\end{equation*}
for $p \geq 1$. Since $\|I_\mu\|_{\Lip^*(
\Theta)} \leq G(\mu)$ and thus the sublevel sets of $G$ are bounded in $\Lip^*(\Theta)$, Banach-Alaoglu theorem applies, implying that the sublevel sets of $G$ are weak$^*$ relatively compact.
The main obstruction to this approach is that the image of the embedding $\mathcal{M}_1(\Theta)\ni\mu\mapsto I_\mu\in \Lip(\Theta)^*$ is not weak$^*$ closed in $\Lip(\Theta)^*$. A possible solution would be to look for generalized minimizers in $\Lip(\Theta)^*$, not necessarily belonging to $\M_1(\Theta)$. This would require extending $G$ to all $\Lip(\Theta)^*$ in a way which is weak$^*$ lower semicontinuous in $\Lip(\Theta)^*$. This is however not immediate, since there is no obvious analogue of the total variation measure $|\mu|$ for an arbitrary element of $\Lip(\Theta)^*$. Additionally, $\Lip(\Theta)$ is typically not separable, which means that the compactness from Banach-Alaoglu does not imply sequential compactness. The lack of separability can be seen just considering the set 
\[\Lip(\R) \supset \{z\mapsto (z-t)^+\}_{t \in (a,b)} \quad \text{with }a,b \in \R,\]
on which no countable set can be dense in the topology of $\|\cdot\|_{\Lip}$, since
\[L\big((\cdot-s)^+ - (\cdot-t)^+\big)=1 \quad\text{for all }s,t \in \R \text{ with }s < t.\]
The following example highlights further the pathologies that can arise:

\begin{example}\label{ex:Lipdualnotconv}With $\Theta=\R$ and $e=0$, define the sequence $\{\mu_n\}_{n\in\N}$ given by $\mu_n := \frac{1}{n}\delta_n \in \M_1(\R),$
for which no subsequence converges weak$^*$ in $\Lip(\R)^*$. To see this, we can use a construction almost identical to the one used in Example \ref{ex:dipolesnotconv}, for which we fix a strictly increasing sequence of indices $n_k$ with $k \in \N$. Taking a not relabelled subsequence, we may always assume that $n_{k+1} \geq 3 n_k$, and based on it define a function $g: \R \to \R$ as
\[g(t):=(-1)^k \left( n_k -\big(|t|- n_k\big)\frac{n_{k+1} + n_k}{n_{k+1}-n_k}\right) \quad \text{if }|t| \in [n_k, n_{k+1}),\]
in which for convenience we denote $n_0 = 0$. Then $g(0)=0$, and by the condition $n_{k+1} \geq 3 n_k$ we have $L(g) = \frac{n_{k+1} + n_k}{n_{k+1}-n_k} \leq 2$, so $g \in \Lip_0(\R)$. But
\[\int_\R g(t) \d \mu_{n_k}(t) = \frac{1}{n_k}g(n_k)= (-1)^k,\]
implying that the original subsequence cannot converge weak$^*$ in $\Lip(\R)^*$. Note that this function has a similar behavior and interplay with $\mu_n$ as in the constructions in \cref{ex:dipolesnotconv}, but at infinity instead of the origin and adapted to the particular choice of indices $n_k$. On the other hand, we have $\|\mu_n\|^* \leq G(\mu_n) \leq 2$, so by Banach-Alaoglu we have that there is a subnet $\{n_{h(i)}\}_{i \in I}$ indexed by some directed set $I$ such that
\[\wstlim_{i \in I} \frac{1}{n_{h(I)}}\delta_{n_{h(I)}} = \varrho \quad \text{ for some }\varrho \in \Lip_0(\R)^*.\]
\end{example}

%%%%%%%%%%%%%%%%%%%%%%%%%%%%%%%%%%%%%%%%%%%%%%%

\section{Compactness in \texorpdfstring{$\KRu$}{KRu}-norm}\label{sec:comp}

%%%%%%%%%%%%%%%%%%%%%%
In this Section~we prove a general compactness result in the KRu norm under total variation and moment bounds, see Theorem~\ref{thm:compact-embedding-locally-compact}. This investigation was motivated by the variational problem \eqref{eq:risk_new} but has an interest in its own right. For this reason, we devote this Section~to presenting the result with its proof and then show its consequences for the optimization problem \eqref{eq:risk_new} in a separate section. 

First, we recall that if $Z$ is compact then the closed unit ball in $\M(Z)$ is compact with respect to the Kantorovich-Rubinstein norm $\|\cdot\|_{\KR}$ (note that for a compact $Z$ the requirement on finite first moment is redundant). This is a consequence of \cite[Thm. VIII.4.3]{kantorovich-akilov}. Therein, only  balanced measures are considered. For the sake of completeness, we provide a proof for the unbalanced setting and the norm $\|\cdot\|_{\KRu}$.

\begin{theorem}\label{thm:compact-embedding-compact}
    Let $Z$ be a compact metric space. Then the following set is compact in $\KRu(Z)$
    $$B = \{\mu \in \M(Z) \colon \norm{\mu}_\TV \leq 1\}.$$
\end{theorem}
\begin{proof}
Let $\{\mu_n\}_{n\in\N}$ be a sequence in $B$. Note that, as in the proof of \cref{thm:dualityKR}, for every $n\in\N$ we can write 
$\mu_n = \mu^0_n + \mu_n(Z) \delta_e,$
    where $\mu_n^0\in\mathcal{M}^0(Z)$.
Therefore, we have that $
   \|\mu_n^0\|_\mathcal{\TV} \leq \|\mu_n\|_\mathcal{\TV} + |\mu_n(Z)| \leq 2$. By \cite[Thm. VIII.4.3]{kantorovich-akilov}, the unit ball $\{\nu \in \M^0(Z) \colon \norm{\nu}_\TV \leq 1\}$ is strongly compact in $\KR(Z)$ and thus there exist a subsequence $\{\mu_{n_k}\}_{k\in\N}$, $\mu^0\in\M^0(Z)$ and $t\in\R$ such that $\mu^0_{n_k} \rightarrow \mu^0$ in the norm $\|\cdot\|_{\KR}$ and $\mu_{n_k}(Z) \rightarrow t$. By definition of the $\KRu$ norm \eqref{eq:KR-norm}, this implies that $\mu_{n_k}\rightarrow\mu^0+t\delta_e$ in $\|\cdot\|_{\KRu}$. It remains to show that $\mu^*=\mu^0+t\delta_e\in B$. This follows by the fact that convergence in $\|\cdot\|_{\KRu}$ implies weak$^*$ convergence, see Lemma~\ref{lem:KR_implies_w*}, together with the lower semicontinuity of the total variation norm with respect to the weak$^*$ convergence. Indeed, we have that 
\begin{align*}
\|\mu^*\|_{\TV}&=\sup_{\overset{f\in C(Z)}{\|f\|_{\infty}\leq 1}} \int_X f (z) {\rm d} \mu^*(z)=\sup_{\overset{f\in C(Z)}{\|f\|_{\infty}\leq 1}} \liminf_{k\to\infty}\int_Z f (z) {\rm d} \mu_{n_k}(z)\\
&\leq \liminf_{k\to\infty} \sup_{\overset{f\in C(Z)}{\|f\|_{\infty}\leq 1}} \int_Z f (z) {\rm d} \mu_{n_k}(z)
=\liminf_{k\to\infty}\|\mu_{n_k}\|_{\TV}\leq1,
\end{align*}
which implies that $\mu^*\in B$ and concludes the proof.
\end{proof}

However, if $Z$ is only locally compact, the situation is subtler and  moment conditions on the measures are necessary to ensure compactness in the KR topology, see Theorem~\ref{thm:compact-embedding-locally-compact}. We need the following definitions and facts.

\begin{definition}[Uniformly tight families of measures]
    A family of measures $M \subset \M(Z)$ is called uniformly tight if for any $\eps > 0$ there exists a compact set $K_\eps \subset Z$ such that $\int_{Z \setminus K_\eps} \d\abs{\mu}(z) < \eps$, for all $\mu \in M$.
\end{definition}

\begin{definition}[Narrow convergence of measures]\label{def:weakconv}
    A sequence $\{\mu_n\}_{n\in\N} \subset \M(Z)$ is said to converge narrowly to $\mu^* \in \M(Z)$, and we write $\mu_n \weakto \mu^*$, if for any  function $\varphi \in C_b(Z)$ one has
    \begin{equation*}
        \int_Z \varphi(z) \, \d\mu_n(z) \to \int_Z \varphi(z) \, \d\mu^*(z),
    \end{equation*}
    where $C_b(Z)$ denotes the space of bounded and continuous real-valued functions on $Z$.
\end{definition}

\begin{theorem}[{Prokhorov's theorem, \cite[Thm. 8.6.2]{bogachev-measure-theory-vol2}}]\label{thm:prok}
    Let $Z$ be a complete separable metric space and $M \subset \M(Z)$ a family of Radon measures on $Z$. Then the following are equivalent
    \begin{itemize}
        \item[(i)] every sequence $\{\mu_n\}_{n\in\N} \subset M$ has a narrowly convergent subsequence;
        \item[(ii)] the family $M$ is uniformly tight and uniformly bounded in the variation norm.
    \end{itemize}
\end{theorem}

\begin{lemma}\label{lem:llss}
        Let $\{\mu_n\}_{n\in\N}$ be a sequence in $\mathcal{M}(Z)$ converging narrowly to $\mu\in \mathcal{M}(Z)$. Then
    \[\liminf_n \int_Z g(x) \, {\rm d} |\mu_n|(x) \geq \int_Z g(x) \, {\rm d} |\mu|(x) \quad \text{for all } g \in C_b(Z),\, g>0.\]
\end{lemma}
\begin{proof} First, we notice that for all $g \in C_b(Z),\, g>0$, and $\nu\in \mathcal{M}(Z)$ we have $\|g\nu\|_{{\rm TV}} = \int_Z g(x) {\rm d} |\nu|$. Thus
    \begin{equation*}
        \begin{aligned}
            &\liminf_n \int_Z g(x) {\rm d} |\mu_n|  = \liminf_n \|g\mu_n\|_{{\rm TV}} = \liminf_n \sup_{\overset{f\in C_0(Z)}{\|f\|_{\infty}\leq 1}} \int_Z f (x) g(x) {\rm d} \mu_n\\
            \geq& \sup_{\overset{f\in C_0(Z)}{\|f\|_{\infty}\leq 1}} \liminf_n \int_Z f (x) g(x) {\rm d} \mu_n\geq \sup_{\overset{f\in C_0(Z)}{\|f\|_{\infty}\leq 1}} \int_Z f (x) g(x) {\rm d} \mu 
            = \|g\mu\|_{{\rm TV}} = \int_Z g(x) {\rm d} |\mu|.
        \end{aligned}
    \end{equation*}
    Notice that we have used the fact that $f g \in C_b(Z)$ when $f \in C_0(Z)$.
\end{proof}

In view of this result, we are now able to prove the lower semicontinuity of the $p$-moments with respect to the narrow convergence. 

\begin{lemma}\label{lem:lsc_qmoments}
    Let $\{\mu_n\}_{n\in\N}$ be a sequence in $\mathcal{M}_p(Z)$ converging narrowly to $\mu\in \mathcal{M}(Z)$. Then
    \[\liminf_n \int_Z (1+d(x,e)^p) {\rm d} |\mu_n|(x) \geq \int_Z (1+d(x,e)^p) {\rm d} |\mu|(x),\]
    and thus $\mu \in \mathcal{M}_p(Z)$.
\end{lemma}
\begin{proof}
    Let $f: x \mapsto  1+d(x,e)^p$ and define the sequence \[ f_k: x \mapsto \min\{f(x),k\} \quad \text{for all } k\in \N.\]
    We have $\lim_k f_k(x) = \sup_k f_k(x) = f(x)$ for all $x\in Z$.
        Using this, the monotone convergence theorem and Lemma~\ref{lem:llss}, we get
        \[\begin{aligned}
            & \liminf_n \int f(x) {\rm d} |\mu_n|(x)   = \liminf_n \sup_k \int f_k(x) {\rm d} |\mu_n|(x) \\ 
            \geq &  \sup_k \liminf_n \int f_k(x) {\rm d} |\mu_n|(x)
            \geq  \sup_k \int f_k(x) {\rm d} |\mu|(x) = \lim_k \int f_k(x) {\rm d} |\mu|(x).
        \end{aligned}\]
        Finally, Fatou's Lemma~implies
        \[\lim_k \int f_k(x) {\rm d} |\mu|(x) \geq \int \liminf_k f_k(x) {\rm d} |\mu|(x) = \int  f(x) {\rm d} |\mu|(x).\]
\end{proof}

With these results, we are now able to prove a general result of compactness in the KRu norm under total variation and moment bounds.
\begin{theorem}\label{thm:compact-embedding-locally-compact}\label{thm:compact}
    Suppose that $p>1$. Then the following set is compact in $\KRu(Z)$
\[
B^p = \left\{\mu\in\mathcal{M}(Z):\int_Z(1+d(z,e)^p){\rm d}|\mu|(z)\leq1\right\}.
\]
\end{theorem}
\begin{proof}
    Consider a sequence $\{\mu_n\}_{n\in\N} \subset \M(Z)$ such that \begin{equation}\label{eq:bddmasspmoment}
        \int_Z (1+d(z,e)^p) \, \d \abs{\mu_n}(z) \leq 1 
    \end{equation}
    uniformly in $n\in\N$. We need to prove that there exist a subsequence $\{\mu_{n_k}\}_{k\in\N}$ and a measure $\mu_\infty \in \M(Z)$ such that
    $\norm{\mu_{n_k} - \mu_\infty}_\KRu \to 0$. 
    Note that, as in the proof of \cref{thm:dualityKR}, we can decompose the elements of the sequence $\{\mu_n\}_{n\in\N}$ as $
    \mu_n = \mu^0_n + \mu_n(Z) \delta_e$, where $\mu_n^0$ is balanced. In particular, we have that 
    \begin{align}
         \int_Z (1+d(z,e)^p) \, \d \abs{\mu^0_n}(z) \leq 2.
    \end{align}
    We start proving that $\{\mu_n^0\}_{n\in\N}$ is uniformly tight. 
Let $\epsilon>0$. Applying the Chebyshev inequality to the function $z\mapsto d(z,e)$ and the measure $|\mu_n^0|$, we obtain that for all $n\in\N$ and for all $r\in(0,+\infty)$
\begin{equation}\label{eq:ptight}\begin{aligned}
    |\mu_n^0|(Z\setminus K_{e,r})&=|\mu_n^0|(\{z\in Z:d(z,e)>r\}) \leq |\mu_n^0|(\{z\in Z:d(z,e)\geq r\})\\
    &\leq \frac{1}{r^p}\int_Zd(z,e)^p{\rm d}|\mu_n^0|(z) \leq\frac{2}{r^p},
\end{aligned}\end{equation}
where we have denoted by $K_{e,r}$ the closed ball centered at $e$ of radius $r$.
Therefore, there exists $r_{\epsilon}>0$ such that $|\mu_n^0|(Z\setminus K_{e,r_\epsilon})<\epsilon$ for all $n\in\N$. Since, by the Hopf-Rinow theorem \cite[Thm.~2.5.28]{burago-burago-ivanov}, the ball $K_{e,r_\epsilon}$ is compact, we can conclude that $\{\mu_n^0\}_{n\in\N}$ is uniformly tight. Since the sequence $\{\mu_n^0\}_{n\in\N}$ is also uniformly bounded in the variation norm by hypothesis, we can apply Prokhorov theorem (\cref{thm:prok}), to conclude that there exists a subsequence $\{\mu_{n_k}\}_{k\in\N}$ narrowly converging to some $\mu_\infty^0\in \mathcal{M}(Z)$, i.e.
\begin{equation}\label{eq:weakconvmu}\int_Z \varphi(z) \d \mu^0_{n_k}(z) \xrightarrow[k \to \infty]{} \int_Z \varphi(z) \d \mu^0_\infty(z) \quad\text{for all } \varphi \in C_b(Z).\end{equation}
Furthermore, by Lemma~\ref{lem:lsc_qmoments}, the limit measure $\mu_\infty^0\in \mathcal{M}(Z)$ satisfies 
   \begin{align*}
         \int_Z (1+d(z,e)^p) \, \d \abs{\mu^0_\infty} \leq 2,
    \end{align*}
    and thus $\mu^0_\infty \in \mathcal{M}_p(Z)$.
Now, let us define $\nu_n \in \mathcal{M}(Z)$ by
\[\nu_n(A) = \int_A d(z,e) \d \mu^0_n(z).\]
Using the Markov (i.e.~Chebyshev with exponent $1$) inequality on the function $z \mapsto d(z,e)^{p-1}$ and the measures $|\nu_n|$, we obtain for all $r \in (0,+\infty)$ and that
\begin{align}\label{eq:mark}
    |\nu_n|(Z \setminus K_{e,r}) \leq \frac{1}{r^{p-1}} \int_Z d(z,e)^{p-1} \d |\nu_n|  = \frac{1}{r^{p-1}} \int_Z d(z,e)^p \d |\mu^0_n| \leq \frac{2}{r^{p-1}},
\end{align}
where we have denoted by $K_{e,r}$ the closed ball centered at $e$ of radius $r$, which is compact by the Hopf-Rinow theorem \cite[Thm.~2.5.28]{burago-burago-ivanov}.

Consider now a function $f \in \Lip_0(Z)$ with Lipschitz constant at most $1$, and let $r \in (0, +\infty)$ be arbitrary. Then we have $|f(z)| \leq r$ for $z \in K_{e,r}$, and we can consider a Tietze extension $\mathcal{E}\big(\restr{f}{K_{e,r}}\big) \in C_b(Z)$ of the restriction $\restr{f}{K_{e,r}}$ of $f$ to the closed set $K_{e,r}$, which inherits the bound \cite[Thm.~VII.5.1, IX.5.2]{dugundji}, that is
\[\left|\mathcal{E}\big(\restr{f}{K_{e,r}}\big)(z)\right| \leq r \quad \text{for all }z \in Z.\]
Now, putting the above facts together, we have
\begin{equation}\label{eq:cutoffintbound}
\begin{aligned}&\left|\int_Z f(z) \d \mu^0_n(z) - \int_Z \mathcal{E}\big(\restr{f}{K_{e,r}}\big)(z) \d \mu^0_n(z)\right| \\ &\qquad\leq \int_{Z \setminus K_{e,r}} |f(z)| \d |\mu^0_n|(z) + \int_{Z \setminus K_{e,r}} \left|\mathcal{E}\big(\restr{f}{K_{e,r}}\big)(z)\right| \d |\mu^0_n|(z) \\ &\qquad\leq \int_{Z \setminus K_{e,r}} d(z,e) \d |\mu^0_n|(z) + \int_{Z \setminus K_{e,r}} r \d |\mu^0_n|(z) \\ &\qquad= |\nu_n|(Z \setminus K_{e,r}) + r |\mu^0_n|(Z \setminus K_{e,r}) \leq \frac{4}{r^{p-1}}.
\end{aligned}
\end{equation}
Similarly, we also obtain
\[\left|\int_Z f(z) \d \mu^0_\infty(z) - \int_Z \mathcal{E}\big(\restr{f}{K_{e,r}}\big)(z) \d \mu^0_\infty(z)\right|\leq \frac{4}{r^{p-1}}.\]
Let $\epsilon>0$, we take $r_\epsilon$ such that $\frac{4}{r_\epsilon^{p-1}}\leq \frac{\epsilon}{3}$. By \eqref{eq:weakconvmu} for $k$ big enough we get
\begin{equation}\label{eq:tietzeconv}\left|\int_Z \mathcal{E}\big(\restr{f}{K_{e,r_\epsilon}}\big)(z) \d \mu^0_{n_k}(z) - \int_Z \mathcal{E}\big(\restr{f}{K_{e,r_\epsilon}}\big)(z) \d \mu^0_\infty(z)\right| \leq \frac{\epsilon}{3}.\end{equation}
By the fact that $\epsilon$ was arbitrary and using triangle inequality, we can conclude that
\[\int_Z f(z) \d \mu^0_{n_k}(z) \xrightarrow[k \to \infty]{} \int_Z f(z) \d \mu^0_\infty(z).\]

Now, we want to improve to convergence in the KRu norm, that is, we would like to obtain convergence with respect of the strong dual topology of $\KR(Z)$ instead of  the weak one. We build upon the proof of \cite[Lem.~VIII.4.4]{kantorovich-akilov}, which treats the case where $Z$ is compact. To this end, for each $\eps > 0$ we will build a finite $\eps$-net for the set of balanced measures in $\{\mu\in\mathcal{M}^0(Z):\int_Z(1+d(z,e)^p){\rm d}|\mu|(z)\leq1\}$, that is, a finite set of measures $\nu_i$ such that for any such $\mu$ there is $i$ for which $\|\mu - \nu_i\|_{\KR} < \eps$. 
Note that this makes such set precompact in the $\KR$ topology, and since we have already seen that it is weakly closed, hence closed, it is compact.

For $r>0$ to be chosen later, let us again consider the closed ball $K_{e,r}$. Since it is compact, we can express it as
\[ K_{e,r} = \bigcup_{j=1}^{m(r,\eps)} E_j \quad \text{for }E_j\text{ Borel sets with }\diam E_j <\frac{\eps}{3}.\]
Moreover, we fix points $z_j \in E_j$ and $z_\infty$ with $r \leq d(e,z_\infty) \leq 2r$ but otherwise arbitrary. Now, for any $\mu \in \mathcal{B}^p$ with $\mu(Z) = 0$ we construct a new measure
\[\overline{\mu} = \sum_{j=1}^{m(r,\eps)} \mu(E_j) \big(\delta_{z_j} - \delta_{e}\big) + \mu\big(Z\setminus K_{e, r}\big)\big(\delta_{z_\infty} - \delta_{e}\big).\]
To estimate the norm of $\mu - \overline{\mu}$, noticing that $\mu(Z)=0$ implies that all the terms containing $\delta_{e}$ above actually cancel, we can write
\begin{align*}
\mu - \overline{\mu} &= \sum_{j=1}^{m(r,\eps)} \rho_j + \rho_\infty, \text{ for } \rho_j, \rho_\infty \text{ defined by}\\
\rho_j(A) &= \mu(A \cap E_j) - \mu(E_j) \delta_{z_j}(A)\\
\rho_\infty(A) &= \mu\big(A \setminus K_{e,r}\big) - \mu\big(Z \setminus K_{e,r}\big) \delta_{z_\infty}(A),
\end{align*}
for which, using \eqref{eq:otformula}, we can estimate $\|\rho_j\|_{\KR} < |\mu(E_j)|\eps/3$ and
\[\|\rho_\infty\|_{\KR} \leq |\mu|\big(Z \setminus K_{e,r}\big) d(e,z_\infty) + \int_{Z \setminus K_{e,r}} d(z,e) \d|\mu|(z) \leq \frac{1}{r^p} \cdot 2r + \frac{1}{r^{p-1}} = \frac{3}{r^{p-1}},\]
where we have used \eqref{eq:ptight} and \eqref{eq:mark}. Since $r$ was still to be chosen, let us fix it as
\[r = \left(\frac{9}{\eps}\right)^{\frac{1}{p-1}}, \quad\text{so that}\quad \|\mu - \overline{\mu}\|_{\KR} < \frac{2\eps}{3}.\]
This shows that we can approximate any balanced $\mu \in B^p$ by a measure supported on finitely many points. To simplify further into a finite set, we need to quantize the amounts that these points are charged with. For this, fix $q \in \Z$ with $q > \frac{3r}{\eps} \big( m(r,\eps) + 2 \big)$, so that $1/q$ is the smallest possible increment of mass. With it we define
\[\overline{\overline{\mu}} = \sum_{j=1}^{m(r,\eps)} \frac{p_j}{q} \big(\delta_{z_j} - \delta_{e}\big) + \frac{p_\infty}{q} \big(\delta_{z_\infty} - \delta_{e}\big),\]
where $p_j, p_\infty$ are the integer parts of $q \mu(E_j)$ and $q \mu\big(Z \setminus K_{e,r}\big)$. Since each term is also balanced and charging only two points, we can estimate
\[\big\|\overline{\mu} - \overline{\overline{\mu}}\big\|_{\KR} \leq \sum_{j=1}^{m(r,\eps)} \frac{1}{q} d(z_j,e) + \frac{1}{q} d(z_\infty,e) \leq m(r,\eps) \frac{r}{q} + \frac{2r}{q} < \frac{\eps}{3},\]
and finally $\big\|\mu - \overline{\overline{\mu}}\big\|_{\KR} < \eps$.

In particular, we obtain that $\mu^0_{n_k}$ converges to $\mu_\infty^0$ with respect to $\|\cdot\|_{\KR}$. Since, up to extracting a further subsequence $\mu_{n_k}(Z) \rightarrow t$ for some $t \in \R$, this implies using \eqref{eq:KRu-directsum} that $\mu_{n_k}=\mu^0_{n_k}+\mu_{n_k}(Z)\delta_e$ converges in $\|\cdot\|_{\KRu}$ to $\mu_\infty:=\mu^0_\infty+t\delta_e$. It remains to show that 
\begin{equation}\label{eq:in_p_ball}
        \int_Z (1+d(z,e)^p) \, \d \abs{\mu_\infty}(z) \leq 1.
    \end{equation}
By \eqref{eq:weakconvmu}, we have that 
\begin{equation}
\int_Z \varphi(z) \d \mu_{n_k}(z) \xrightarrow[k \to \infty]{} \int_Z \varphi(z) \d \mu^0_\infty(z)+t\varphi(e) \quad\text{for all } \varphi \in C_b(Z).\end{equation}
This together with Lemma~\ref{lem:lsc_qmoments} yields \eqref{eq:in_p_ball}, which concludes the proof.
\end{proof}

\begin{remark}
We point out that the separability assumption could be replaced by assuming $Z$ to be connected. On the one hand, by Stone's theorem (see \cite[Thm.~IX.5.2]{dugundji}) all metric spaces are paracompact, while a connected, locally compact and paracompact topological space is also $\sigma$-compact \cite[pp.~460]{Spi79}, which in turn implies \cite[Thm.~XI.7.2]{dugundji} that $Z$ is Lindel\"of. Being a metric space, the Lindel\"of property is equivalent \cite[Thm.~IX.5.6]{dugundji} to $Z$ being separable.
\end{remark}

\paragraph{Lack of compactness for parameter spaces which are not locally compact}
The following simple example shows that in the absence of local compactness an analogue of \cref{thm:compact-embedding-locally-compact} does not hold.
\begin{example}
    Let $H$ be an infinite-dimensional Hilbert space and $Z= \{ z \in H \colon \|z\|_{H} \leq 1\}$ be the closed unit ball equipped with the norm-induced metric. Let $\{z_j\}_{j\in\N} \subset Z$ be an orthonormal sequence and $\mu_j \defeq \delta_{z_j}$ be the sequence of Dirac deltas placed at $z_j$. Choosing $e=0$ as the base point in $Z$, we have for all $j$ that
\begin{equation*}
    \int_Z (1+d(z,e)^p) \, \d \abs{\mu_j}(z) \leq 2 \int_Z \d \abs{\mu_j}(z) = 2.
\end{equation*}
However, for any $j \neq k$ we have $\norm{\mu_j - \mu_k}_\KRu = \abs{(\mu_j - \mu_k)(Z)} +  d(z_j,z_k) = \sqrt{2}$, and there is no convergent subsequence.
\end{example}

Of course, one can equip $Z$ with the weak$^*$ topology and, if $H$ is separable, metrize it to obtain a compact metric space. However, in this case one can only consider pairings with weak$^*$ continuous functions, which can be restrictive. This scenario has been studied in~\cite{korolev2022two}.

%%%%%%%%%%%%%%%%%%%%%%%%%%%%%%%%%%%%%
\section{Regularized minimization problems in \texorpdfstring{$\KRu$}{KRu}}\label{sec:regprobs}

In this Section~we define the optimization problems we will consider in our analysis. Such problems generalize the $\KRu$-regularized risk minimization problems introduced in \eqref{eq:risk_new}.
We consider the following family of problems for $\alpha \geq 0$ and $\beta >0$:
\begin{equation}\label{eq:opt-prob}
    \inf_{\mu \in\KRu(Z)}F(\mathcal{A}\mu)+ G_{\alpha,\beta}(\mu),
\end{equation}
where $(Z,d)$ is a complete, separable, locally compact, geodesic and pointed metric space with base point $e \in Z$, the operator $ \mathcal{A} : \KRu(Z) \to H$ is linear and continuous with values in the Hilbert space $H$, the functional
$ F : H \to (-\infty,+\infty] $
is proper, convex, coercive and lower semicontinuous,
and $G_{\alpha,\beta} \colon \text{KR}(Z) \rightarrow [0,\infty]$ is defined analogously to~\eqref{eq:G-definition-Rd}:
\begin{equation}\label{eq:G-definition}
     G_{\alpha,\beta}(\mu) \defeq \left\{ 
     \begin{array}{ll}
         \alpha\|\mu\|_{\KRu} + \beta \displaystyle \int_Z ( 1 + d(z,e)^p ) \, \d\abs{\mu}(z)   &  \text{if }\mu \in \mathcal{M}_p(Z),\\
          +\infty & \text{otherwise} 
     \end{array}
     \right.
 \end{equation}
 for $p \geq 1$.
To recover the ERM problem \eqref{eq:risk_new} one can consider $\Theta \subset \R^{d+1}$ to be a set of weights with base point $e \in \Theta$ and define the operator $ \mathcal{A} : \KRu(\Theta) \to \R^N$ as
\begin{align}\label{eq:A}
(\mathcal{A} \mu)_i = \scal{\mu}{\sigma(\langle \cdot , x_i\rangle)}{\KRu(\Theta)}{\Lip(\Theta)},
\end{align}
for a collection of data points $\{x_1,\ldots, x_N\} \subset \R^{d+1}$ and any Lipschitz activation function $\sigma$. It can be readily verified that such $\mathcal{A}$ is linear and continuous. Finally, the fidelity term  $F: \R^N \to (-\infty,+\infty]$ can be chosen of the form 
\begin{align}\label{eq:F}
F(w_1,\ldots,w_N)=\frac{1}{N} \sum_{i=1}^N L(y_i,w_i) \quad \text{ with }\{y_1, \ldots, y_N\} \subset \R, 
\end{align}
for a loss $L$ that is convex, coercive and lower semicontinuous in the second entry.

\begin{remark}
The choice of $p$ will  be crucial for the well-posedness properties of the variational problem \eqref{eq:opt-prob}. In particular, we will show that if $p>1$ the regularizer $G_{\alpha,\beta}$ enforces compactness on \eqref{eq:opt-prob}, which allows us to prove existence of minimizers as a consequence of the direct method of calculus of variations. On the other hand, if $p=1$, we lose any a priori compactness and the effect on the regularizer is not enough to ensure existence of minimizers.
\end{remark}

\subsection{Well-posedness for \texorpdfstring{$p>1$}{p>1}}

We now show that by choosing $p>1$ in the definition of $G_{\alpha,\beta}$, the minimization problem \eqref{eq:opt-prob} and thus also the risk minimization problem in \eqref{eq:risk_new} admit solutions. A crucial role will be played by Theorem~\ref{thm:compact-embedding-locally-compact} ensuring compactness of the sublevel sets of 
$G_{\alpha,\beta}$.

\begin{theorem}\label{thm:lsc}
    If $p>1$, the regularizer $G_{\alpha,\beta}$ is \lsc{} w.r.t. the $\KRu$ norm.
\end{theorem}
\begin{proof}
By Theorem~\ref{thm:compact-embedding-locally-compact}, for all $t\in\R$ the $t$-sublevel set of $G_{0,\beta}\colon \KRu(Z) \rightarrow [0,\infty]$
\[
\{\mu\in \KRu(Z):G_{0,\beta}(\mu)\leq t\}=\left\{\mu\in\mathcal{M}_p(Z):\int_Z(1+d(z,e)^p){\rm d}|\mu|(z)\leq \frac{t}{\beta}\right\}
\]
is compact and consequently also closed in $\KRu(Z)$. This is equivalent to $G_{0,\beta}$ being lower semicontinuous w.r.t. the $\KRu$ norm. Furthermore, $G_{\alpha,0}\colon \KRu(Z) \rightarrow [0,\infty]$ is continuous with respect to the $\rm KRu$ convergence, and the proof follows.
\end{proof}

\begin{proposition}\label{prop:G-compact-KR}
    If $p >1$, the sublevel sets of $G_{\alpha,\beta}$ are compact in the $\KRu$ topology.
\end{proposition}
\begin{proof}
Theorem~\ref{thm:compact-embedding-locally-compact} together with the inclusion
   \begin{align*}
       &\{\mu\in \KRu(Z):G_{\alpha,\beta}(\mu)\leq t\} \\
       \subseteq &\{\mu\in \KRu(Z):G_{\alpha,0}(\mu)\leq t\}\cap\{\mu\in \KRu(Z):G_{0,\beta}(\mu)\leq t\}
   \end{align*}
imply that the sublevel sets of $G_{\alpha,\beta}$ are compact in the $\KRu$ topology as, due to \cref{thm:lsc}, are closed subsets of compact sets.
\end{proof}

\begin{theorem}\label{thm:existence_of_minimizers}
    If $p>1$, the optimization problem~\eqref{eq:opt-prob} admits a minimizer.
\end{theorem}
\begin{proof}
    Let $\{\mu_n\}_{n\in\N} \subset \M_p(Z)$ be a minimizing sequence for~\eqref{eq:opt-prob}. Note that $F\circ \mathcal{A}$ is bounded from below and in particular this means that there exist $N\in\N$ and $M>0$ such that  $G_{\alpha,\beta}(\mu_n)\leq M$ for all $n\geq N$, and \cref{prop:G-compact-KR} guarantees the existence of a convergent subsequence $\{\mu_{n_k}\}_{k\in\N}$ such that $\norm{\mu_{n_k} - \mu^*}_{\KRu} \to 0$ for some $\mu^* \in \M_p(Z)$. Since $\calA \colon \
    \KRu(Z) \to H$ is continuous, $F$ is \lsc{} on $H$ and, due to \cref{thm:lsc},  $G_{\alpha,\beta}$ is \lsc{} on $\KRu(Z)$,  the function $F\circ \mathcal{A}+ G_{\alpha,\beta}: \KRu(Z)\to (-\infty , +\infty]$ is lower semicontinuous on $\KRu(Z)$. Thus, 
    \[F(\mathcal{A}\mu^*)+ G_{\alpha,\beta}(\mu^*)\leq \liminf_k F(\mathcal{A}\mu_{n_k})+ G_{\alpha,\beta}(\mu_{n_k}) = \inf_{\mu\in \KRu(Z)} F(\mathcal{A}\mu)+ G_{\alpha,\beta}(\mu), \]
and $\mu^*$ is a minimizer of~\eqref{eq:opt-prob}     
\end{proof}

\subsection{(Non)existence of minimizing measures when \texorpdfstring{$p=1$}{p=1}}
\label{sec:nonexistence}

\textbf{Lack of compactness of the regularizer.}
While in the previous Section~we showed that the variational problem \eqref{eq:opt-prob} is well-posed in the case $p >1$, if $p=1$ it is not difficult to show that the sublevel sets of $G_{\alpha,\beta}$ are not necessarily compact. 

\begin{example}\label{ex:remark1}
Given a metric space $(Z,d)$ with base point $e \in Z$ we can consider the sequence of empirical measures $\mu_n = \frac{\delta_{z_n}}{d(z_n,e)} \in \mathcal{M}_1(Z)$ with $z_n \in Z$ and $d(z_n,e) \rightarrow +\infty$. It holds that 
\begin{align*}
    \|\mu_n\|_{\KRu} = \frac{1}{d(z_n,e)} +  \sup \left\{\int_Z f \, \d\mu_n \colon f(e) = 0, \, L(f) \leq 1\right\} =  \frac{1}{d(z_n,e)} + 1
\end{align*}
and thus for $n$ big enough
\begin{align*}
G_{\alpha,\beta}(\mu_n) \leq  2\alpha + \beta\int_Z (1 + d(z,e)) \, \d\frac{\delta_{z_n}}{d(z_n,e)}(z)  = 2\alpha+ \beta\frac{1 + d(z_n,e)}{d(z_n,e)} \leq C(\alpha,\beta)   
\end{align*}
for a positive constant $C(\alpha,\beta)   $ not depending on $n$. Note that $\mu_n$ converges to the zero measure in total variation, while for the $1$-Lipschitz test function $\varphi(z) = d(z,e)$ it holds that
$\int_Z \varphi(z) \, \d\mu_n(z) = 1$. 
This implies, in particular, the non-existence of the weak KRu limit along any subsequence of $\mu_n$.

\end{example}
This lack of compactness translates to minimizing sequences of \eqref{eq:opt-prob}. This is shown in the following example.
\begin{example}
Given a metric space $(Z,d)$ with base point $e \in Z$ consider the following minimization problem 
\begin{align*}
    \inf_{\mu \in \KRu(Z)} \Big | \int_{Z} d(z,e)\, d\mu(z) - 1\Big|^2 + G_{1/2,1/2}(\mu)
\end{align*}
that is a specific instance of \eqref{eq:opt-prob} for $H = \R$, $\mathcal{A}\mu = \scal{\mu}{d(\cdot,e)}{\KRu(Z)}{\Lip_0(Z)}$, $F(t) = |t-1|^2$ and $\alpha = \beta = \frac{1}{2}$. 
Note that 
\begin{align*}
    &\Big | \int_{Z} d(z,e)\, \d\mu(z) - 1\Big|^2 + G_{1/2,1/2}(\mu)\\
    \geq & \Big | \int_{Z} d(z,e)\, \d\mu(z) - 1\Big|^2 + \frac{1}{2}\int_{Z} (1 + d(z,e))\, \d|\mu|(z) + \frac{1}{2}\|\mu\|_{\KRu} \\
    \geq & \Big | \int_Z d(z,e)\, \d\mu(z)\Big|^2 + 1 - 2 \int_Z d(z,e)\, \d\mu(z) + \frac{1}{2}\int d(z,e)\, \d|\mu|(z) + \frac{1}{2}\|\mu\|_{\KRu}\\
    \geq &  \Big | \int_Z d(z,e)\, \d\mu(z)\Big|^2 + 1 - \int_Z d(z,e)\, \d\mu(z) \geq \frac{3}{4}
    %& = \Big | \int_X d(x,e)\, \d\mu(x)\Big|^2 + 1 - 2 \int_X d(x,e)\, \d(\mu^+ - \mu^-)(x) + \int d(x,e)\, \d(\mu^+ + \mu^-)(x
\end{align*}
for all $\mu \in \KRu(Z)$. 
Choosing any sequence $z_n \in Z$ such that $d(z_n,e) \rightarrow +\infty$ consider now the sequence of measures $\mu_n = \frac{\delta_{z_n}}{2d(z_n,e)} \in \mathcal{M}_1(Z)$ and observe that  $  | \int_{Z} 
d(z,e)\, \d\mu_n(z) - 1|^2 = \frac{1}{4}$ and 
\begin{align*}
   \lim_{n\rightarrow +\infty} \int_{Z} (1 + d(z,e))\, \d|\mu_n|(z) = \lim_{n\rightarrow +\infty} \frac{1}{2d(z_n,e)} + \frac{1}{2} = \frac{1}{2}.
\end{align*}
Moreover $\|\mu_n\|_{\KRu} \rightarrow \frac{1}{2}$ as $n\rightarrow +\infty$. Indeed $|\mu_n(Z)| = \frac{1}{2d(z_n,e)} \rightarrow 0$ and for every $f \in {\rm Lip}_0(Z)$ such that $L(f) \leq 1$ it holds that 
\begin{align*}
    \int_{Z} f(z)\, \d|\mu_n|(z) = \frac{f(z_n)}{2d(z_n,e)} \leq \frac{1}{2} 
\end{align*}
with equality achieved for $f(z) = d(z,e)$.
Therefore 
\begin{align*}
\lim_{n\rightarrow +\infty}\Big | \int_{Z} d(z,e)\, & \d\mu_n(z) - 1\Big|^2 + G_{1/2,1/2}(\mu_n) = \frac{1}{4} +\frac{1}{2} = \frac{3}{4}. 
\end{align*}
This implies that the sequence $\mu_n$ is a minimizing sequence for \eqref{eq:opt-prob} that does not admit a weak converging subsequence in KRu as shown in Example \ref{ex:remark1}.
\end{example}

\paragraph*{Pushing mass to infinity prevents existence of nontrivial minimizers} 
Building on this reasoning it is possible to show that for natural (but non-compact) choices of the weight space $\Theta \subset \R^{d+1}$, the $\KRu$-regularized ERM problem \eqref{eq:risk_KR} with $p=1$ is ill-posed.
We start with a lemma that deals with a slightly more general setting than the one needed for  $\KRu$-regularized ERM problems.

\begin{lemma}\label{lem:nonex}
    For a Hilbert space $H$, let $\mathcal{A} : \KRu(\R^{d+1}) \rightarrow H$ defined as 
    \begin{align}\label{eq:opp}
        \mathcal{A}\mu = \scal{\mu}{\phi}{\KRu(\R^{d+1})}{\Lip(\R^{d+1})}
    \end{align}
    for $\varphi : \R^{d+1} \rightarrow H$ that is Lipschitz and $1$-homogeneous. Let $(Z,d)$ be the space $\R^{d+1}$ endowed with the Euclidean distance and the origin as a base point. Then \eqref{eq:opt-prob} with $p=1$ either has $\mu =0$ as minimizer or it does not admit a minimizer.
\end{lemma}
\begin{proof}
Define a rescaling $\Phi_R(z)=Rz$ with $R >1$, and for any $\mu \in \M_1(\R^{d+1})$ the measure $\mu_R = \frac{1}{R}\big(\Phi_R\big)_\# \mu$.
Then by the homogeneity of $\varphi$ we have
\begin{align*}
& \mathcal{A} \mu_R = \int_{\R^{d+1}} \varphi(z)\, \d\mu_R(z) = \int_{\R^{d+1}} \varphi(z)\, \d\mu(z) \quad \text{  and  }\\
& \int_{\R^{d+1}} \|z\|\, \d|\mu_R|(z) = \int_{\R^{d+1}} \|z\|\, \d|\mu|_R(z) = \int_{\R^{d+1}} \|z\|\, \d|\mu|(z),
\end{align*}
while $|\mu_R|(\R^{d+1})=|\mu|(\R^{d+1})/R < |\mu|(\R^{d+1})$ if $\mu \neq 0$.
Moreover, for every $f$ such that $f(e) = 0$ and $L(f) \leq 1$ it holds that 
\begin{align*}
    \int_{\R^{d+1}} f \, \d\mu_R = \frac{1}{R} \int_{\R^{d+1}} f(Rz) \, \d\mu \leq \int_{\R^{d+1}} d(e,z) \, \d\mu \leq \|\mu\|_{\KRu}, 
\end{align*}
implying that $\|\mu_R\|_{\KRu} \leq \|\mu\|_{\KRu}$. In particular if $\mu \neq 0$, then 
\begin{align*}
    F(\mathcal{A}\mu_R)+ G_{\alpha,\beta}(\mu_R) <  F(\mathcal{A}\mu)+ G_{\alpha,\beta}(\mu)
\end{align*}
showing that no $\mu \neq 0$ can be a minimizer.    
\end{proof}

Given a set of data and labels $(x_i,y_i)_{i=1}^N \subset \R^{d+1} \times \R$, let us choose now as set of weights $\Theta = \R^{d+1}$. Define also the function $\varphi: \R^{d+1} \rightarrow \R^N$ as
    \begin{align}
        (\varphi(\theta))_i = \sigma(\langle \theta,x_i\rangle) \quad i = 1,\ldots, N,
    \end{align}
    where $\sigma$ is a Lipschitz, $1$-homogeneous activation function, and the operator $\mathcal{A} : \KRu(\R^{d+1}) \rightarrow H$ is as in \eqref{eq:opp}.
    By choosing the fidelity term  $F: \R^N \to (-\infty,+\infty]$ as in \eqref{eq:F}
for a loss $L$ that is convex, coercive and lower semicontinuous in the second entry, Problem \eqref{eq:opt-prob} is the $\KRu$-regularized ERM problem \eqref{eq:risk_KR} for measures on $\Theta = \R^{d+1}$ and activation functions $\sigma$ that are Lipschitz and $1$-homogeneous.
Note, in particular, that ReLU and Leaky ReLU activation functions enter in such class. By applying Lemma~\ref{lem:nonex} we deduce that such $\KRu$-regularized ERM problems with $p=1$ have either $\mu = 0$ as minimizer or do not admit minimizers. For many choices for the labels $y_i$ the case that  $\mu = 0$ is a minimizer can be excluded and it is not very interesting in practice as well, leading to the ill-posedness of \eqref{eq:opt-prob}. 

\paragraph*{An example of existence without compactness} While for $p=1$ no compactness is, in general, available, in many cases minimizers still exist. Here we provide an example for a very specific case. We consider $Z = \mathbb{R}$ endowed with the metric $d(z,w) = |z-w|$ and the linear operator $\mathcal{A}\mu = \scal{\mu}{\varphi}{\KRu(\R)}{\Lip(\R)}$
where $\varphi : \R \rightarrow \R$ is a Lipschitz function defined as follows: $\varphi(z) = |z|$ for $z \geq 0$ and $\varphi(z)$ arbitrary and bounded for $z <0$. Choosing as fidelity term $F(t) = |t-y|$ for an arbitrary $y \in \R$, we consider the following minimization problem, that is a specific instance of \eqref{eq:opt-prob}:
\begin{align}\label{eq:innn}
    \inf_{\mu \in \KRu(\R)} \left| \int_\R \varphi(z)\, d\mu - y\right| + G_{0,1}(\mu).
\end{align}
We now prove that given $\mu \in \mathcal{M}_1(\R)$ it holds that $\mu \mres \R_-$ has lower energy.
Note that
\begin{align*}
    J(\mu) & :=\Bigg|\int_\R \varphi(z)\, \d\mu - y\Bigg|  + \int_\R (1 + |z|)\, \d|\mu| \\
    & = \left| \int_{\R_-} \varphi(z)\, \d\mu + \int_{\R_+} |z|\, \d\mu - y\right| + \int_{\R_-} (1 + |z|)\, \d|\mu| + \int_{\R_+} (1 + |z|)\, \d|\mu|  \\
    & \geq \left| \int_{\R_-} \varphi(z)\, d\mu + \int_{\R_+} |z|\, d\mu - y\right| + \int_{\R_+} |z|\, \d|\mu| + \int_{\R_-} (1 + |z|)\, \d|\mu| \\
    & = \left| \int_{\R_-} \varphi(z)\, d\mu + \int_{\R_+} |z|\, \d\mu^+ -  \int_{\R_+} |z|\, \d\mu^- - y\right| \\&\qquad+ \int_{\R_+} |z|\, \d\mu^+ + \int_{\R_+} |z|\, d\mu^-  + \int_{\R_-} (1 + |z|)\, d|\mu|.
\end{align*}
We now minimize separately in $a:=\int_{\R_+} |z|\, d\mu^+$ and $b:=\int_{\R_+} |z|\, d\mu^-$. Note that we can do that since $a$ is independent on $b$. We thus estimate the previous expression by
\begin{align*}
& \geq \min_{a \geq 0} \left| \int_{\R_-} \varphi(z)\, \d\mu + a -  \int_{\R_+} |z|\, \d\mu^- - y\right| + a + \int_{\R_+} |z|\, \d\mu^-  + \int_{\R_-} (1 + |z|)\, \d|\mu|\\
 & = \left| \int_{\R_-} \varphi(z)\, \d\mu -  \int_{\R_+} |z|\, \d\mu^- - y\right| + \int_{\R_+} |z|\, \d\mu^-  + \int_{\R_-} (1 + |z|)\, \d|\mu|\\
& \geq \min_{b \geq 0} \left|\int_{\R_-} \varphi(z)\, \d\mu -  b - y\right| + b  + \int_{\R_-} (1 + |z|)\, \d|\mu|\\
& = \left|\int_{\R_-} \varphi(z)\, \d\mu- y\right|  + \int_{\R_-} (1 + |z|)\, \d|\mu| = J(\mu \mres \R_-).
\end{align*}
In particular, this implies that every minimizing sequence of $J$, denoted by $\mu_n$, can be chosen to be supported in $\R_-$. Therefore considering the problem
\begin{align}\label{eq:miii}
     \inf_{\mu \in \mathcal{M}_1(\R)}   \left| \int_\R \frac{\varphi(z)}{1+|z|}\, \d\mu - y\right| + |\mu|(\R)
\end{align}
one can observe that if $\mu_n$ is a minimizing sequence of $J$ supported on $\R_-$, then $\frac{\mu_n}{1+|z|}$ is a minimizing sequence for  \eqref{eq:miii} again supported on $\R_-$. Therefore by weak$^*$ compactness of measures bounded in total variation we conclude existence of minimizers both for \eqref{eq:miii} and for \eqref{eq:innn}.

\section{Spaces of continuous functions with controlled growth}\label{sec:controlledgrowth}
A limitation of the approach in Section~\ref{sec:two} is that it does not apply to activation functions which are not Lipschitz, for instance, to Rectified Power Unit (RePU) activation functions $\sigma_m=\max\{0,\cdot\}^{m-1}$ with $m>2$. In this section, we present an alternative setting by interpreting the integral \eqref{eq:infinite_width_nn} as a dual pairing 
\begin{equation}\label{eq:dualparingweight}
\scal{\mu}{\sigma(\langle\cdot,x\rangle)}{C_w(\Theta)^*}{C_w(\Theta)}
\end{equation}
between a weighted space of continuous functions and its topological dual. Spaces of continuous functions with controlled growth at infinity provide a large reservoir of activation functions and, in particular, they include RePU. Furthermore, the topological duals of weighted spaces of continuous functions can be characterized as spaces of Radon measures which can be factored into the product of the weight with a bounded Radon measure \cite{summers1970dual}. 

\subsection{Weighted spaces of continuous functions and their dual spaces.}

Let $w\colon Z\to (0,+\infty)$ be a positive continuous function and consider the space $C_w(Z)$ of continuous functions $\zeta\colon Z\to\R$ satisfying 
\begin{equation}\label{eq:q_growing_condition}
    \lim_{d(z,e)\to\infty}\zeta(z)w(z)=0,
\end{equation}
endowed with the norm
$\|\zeta\|_{C_w(Z)}=\sup_{z\in Z}|\zeta(z)|w(z)$. The space $C_w(Z)$ is complete. The proof follows the lines of the one for \cite[Lemma~3.2]{NalSav21}, once noticing that equation~\eqref{eq:q_growing_condition} is equivalent to the following condition:
\begin{equation}\label{eq:q_growing_condition2}
\forall \epsilon>0\quad \exists A_\epsilon>0\, :\, |\zeta(z)|\leq A_\epsilon+\epsilon w(z)^{-1}\quad \forall z\in Z. 
\end{equation}

If $w\colon Z\to(0,+\infty)$ is a positive continuous function on $Z$, we use the notation
\[
w\cdot\mathcal{M}(Z)=\{w\cdot\mu:\mu\in\mathcal{M}(Z)\},
\]
where 
\[
\int_{Z}f(z){\rm d}(w\cdot\mu)(z)=\int_Z f(z)w(z){\rm d}\mu(z),\quad \forall f\in C_c(Z).
\]

It is immediate to verify that the space of measures $w\cdot\mathcal{M}(Z)$ embeds into the topological dual of $C_w(Z)$. Indeed, for every $\nu\in w\cdot\mathcal{M}(Z)$, the linear functional $T_\nu\colon C_w(Z)\to\mathbb{\R}$ defined as 
\[
T_\nu(\zeta)=\int_Z \zeta(z){\rm d}\nu(z)
\]
is continuous. This can be easily derived with the following computation:
\begin{equation*}
    \begin{aligned}
        |T_\nu(\zeta)|&\leq \int_Z|\zeta(z)|{\rm d}|\nu|(z)=\int_Z\frac{|\zeta(z)|}{w(z)}w(z){\rm d}|\nu|(z)\\
    &\leq\|\zeta\|_{C_w(Z)}\int_Z w(z)^{-1}{\rm d}|\nu|(z)=\|\zeta\|_{C_w(Z)}\|w^{-1}\cdot\nu\|_{\rm TV}.
    \end{aligned}
\end{equation*}
Actually, Theorem~3.1 in \cite{summers1970dual} proves that $(C_w(Z))^* \cong w \cdot \mathcal{M}(Z)$,
where $w \cdot \mathcal{M}(Z)$ is endowed with the dual norm
\begin{equation}\label{eq:dualnorm}
\begin{aligned}
\|\nu\|_{w,*}& 
=\sup_{\underset{\|\zeta\|_{C_w(Z)}\leq 1}{\zeta\in C_w(Z)}}\left\{\int_Z\zeta(z){\rm d}\nu(z)\right\}
=\sup_{\underset{\|\zeta\|_{C_w(Z)}\leq 1}{\zeta\in C_w(Z)}}\left\{\int_Z\zeta(z)w(z){\rm d}(w^{-1}\cdot\nu)(z)\right\}\\ &=\sup_{\underset{\|\varphi\|_{\infty}\leq 1}{\varphi\in C_0(Z)}}\left\{\int_Z\varphi(z){\rm d}(w^{-1}\cdot\nu)(z)\right\} =\|w^{-1}\cdot\nu\|_{\rm TV}
=\int_Z w(z)^{-1}{\rm d}|\nu|(z).
\end{aligned}
\end{equation}
Then, for any continuous linear functional $T$ on $C_w(Z)$, there exists a unique measure $\nu\in w\cdot\mathcal{M}(Z)$ such that $T=T_\nu$.
In particular, if we choose as weight the function $w_q=(1+d(\cdot,e)^q)^{-1}$, with $q\in (0,+\infty)$, the space $w_q\cdot\mathcal{M}(Z)$ coincides with the subspace $\mathcal{M}_q(Z)$ of measures with finite $q$-moments, see~\eqref{eq:def_of_Mp}. 

\begin{lemma}\label{lem:identification}
    Let $q\in(0,+\infty)$ and $w_q=(1+d(\cdot,e)^q)^{-1}$. Then, we have that
    \begin{equation}\label{eq:momentsweights}\mathcal{M}_q(Z)=w_q\cdot\mathcal{M}(Z).\end{equation}
\end{lemma}
\begin{proof}
    Clearly, $w_q\cdot\mathcal{M}(Z)\subseteq \mathcal{M}_q(Z)$. Indeed, for any $\mu\in\mathcal{M}(Z)$
    \begin{align*}
        \int_Z d(z,e)^q {\rm d}|w_q\cdot\mu|=\int_Z w_q(z)d(z,e)^q{\rm d}|\mu|=
        \int_Z\frac{d(z,e)^q}{1+d(z,e)^q}{\rm d}|\mu|
        \leq C_q\|\mu\|_{\TV},
    \end{align*}
    with $C_q=\max\{w_q(z)d(z,e)^q:z\in Z\}$. To show the opposite inclusion $\mathcal{M}_q(Z)\subseteq w_q\cdot\mathcal{M}(Z)$, let $\nu\in\mathcal{M}_q(Z)$. Without loss of generality, we can assume that 
    $$\int_Zw_q(z)^{-1}{\rm d}|\nu|(z)=\int_Z(1+d(z,e)^q){\rm d}|\nu|(z)\leq 1.$$
Then, for every $f\in C_c^+(Z)$ with $\|f\|_{C_{w_q}}\leq 1$
\begin{align*}
\int_Z f(z){\rm d}\nu^+(z)=\int_Z \frac{f(z)}{w_q(z)}w_q(z){\rm d}\nu^+(z)\leq\|f\|_{C_{w_q}}\int_Z w_q(z)^{-1}{\rm d}\nu^+(z)\leq1.
\end{align*}
Then, by Lemma~3.3 in \cite{summers1970dual}, there exists a non-negative measure $\mu^+\in\mathcal{M}(Z)$ with $\|\mu^+\|_{\rm TV}\leq 1$ such that $\nu^+=w_q\cdot\mu^+$. Analogously, there exists $\mu^-\in\mathcal{M}(Z)$ with $\|\mu^-\|_{\rm TV}\leq 1$ such that $\nu^-=w_q\cdot\mu^-$. Therefore, $\nu=\nu^++\nu^-=w_q\cdot(\mu^++\mu^-)$, which shows that $\nu\in w_q\cdot\mathcal{M}(Z)$.
% and concludes the proof.
\end{proof}
Hence, the topological dual of $C_{w_q}(Z)$ coincides with $\mathcal{M}_q(Z)$ endowed with the norm
\[
\|\nu\|_{\mathcal{M}_q}:=\|\nu\|_{w_q,*}=\int_Z(1+d(z,e)^q){\rm d}|\nu|(z),
%=\|\nu\|_{\rm TV}+\int_Z d(z,e)^q{\rm d}|\nu|(z),
%
\]
which is precisely the regularizer studied in the previous sections.

\begin{remark}
    For $q>1$, the weak$^*$ convergence we considered in $\mathcal{M}_q(Z)$ (let us call this topology $\tau_q^*$) implies the $\KRu$ convergence. For any $p>1$, let \[B^p :=\left\{\mu\in\mathcal{M}(Z):\int_Z(1+d(z,e)^p){\rm d}|\mu|(z)\leq1\right\}.\] We have shown in Theorem~\ref{thm:compact-embedding-locally-compact} that this set is compact with respect to the $\KRu$ topology and thus, in this set, the $\KRu$ topology coincides with its correspondent weak topology, i.e., the weakest topology that makes continuous the pairing against Lipschitz functions (let us call this topology $\tau_L$). To see this, just notice that the identity map is strong to weak continuous, a continuous bijection from a compact space to a Hausdorff one must be a homeomorphism \cite[Thm.~IX.2.1]{dugundji}, and the weak topology in $\KRu(Z)$ is Hausdorff by applying the Hahn-Banach separation theorem on its predual $\Lip(Z)$. Let $\{\mu_n\}_n$ be a sequence converging to some $\mu$ w.r.t. the weak$^*$ topology in $\mathcal{M}_q(Z)$ for some $q>1$. Then, $\|\mu_n\|_{\mathcal{M}_q(Z)}\leq c$, for some $c > 0$, and for every $p$ such that $1<p<q$ we have 
    \begin{align*}
        \int (1+d(z,e)^p) \, {\rm d} |\mu_n| & = _{C_{w_q}}\langle 1+d(\cdot,e)^p, |\mu_n| \rangle_{\mathcal{M}_q}  
        \leq  \|1+d(\cdot, e)^p\|_{C_{w_q}(Z)}\|\mu_n\|_{\mathcal{M}_q} \\
        & =  \|1+d(\cdot, e)^p\|_{C_{w_q}}\|\mu_n\|_{\mathcal{M}_q(Z)}\leq c \|1+d(\cdot, e)^p\|_{C_{w_q}}.
    \end{align*}
     This means that $\{\mu_n\}_n\subset \gamma B^p$ for some $\gamma  > 0$. From Lemma~\ref{lem:lsc_qmoments} also $\mu\in \gamma B^p$. Since $\KRu$ coincides with $\tau_L$ in $B^p$ (and  also in $\gamma B^p$) and convergence of sequences in $(\mathcal{M}_q,\tau_q^*)$ implies convergence  w.r.t. the $\tau_L$ topology (since $q > 1$), we conclude. 
    \end{remark}
    \begin{remark}
It is also possible to show that the topology $\KRu$ coincides with $\tau_p^*$ in $B^p$, when $p > 1$. In fact, $B^p$ is the ball in $\mathcal{M}_p(Z)$ and thus, it is compact w.r.t. its weak$^*$ topology $\tau_p^*$. Since $p > 1$, the identity map $I:(B^p,\tau_p^*)\to (B^p,\tau_L)$ is continuous. We notice that $(B^p,\tau_L)$ is Hausdorff by Hahn-Banach separation theorem and we conclude again with recalling that a continuous bijection from a compact space to a Hausdorff one must be a homeomorphism \cite[Thm.~IX.2.1]{dugundji}. 
\end{remark}

\subsection{Connection to KRu- and TV-regularized empirical risk minimization}

Also in this setting, for a set of weights $\Theta \subset \R^{d+1}$, we can consider the function spaces defined by the dual pairings  
\[
\scal{\sigma(\langle\cdot,x\rangle)}{\nu}{C_w(\Theta)}{C_w(\Theta)^*}=\scal{\sigma(\langle\cdot,x\rangle)}{\nu}{C_w(\Theta)}{w\cdot\mathcal{M}(\Theta)},
\]
as well as the associated ERM problems. Let $w\colon \Theta\to(0,+\infty)$ be a positive continuous function on $\Theta$. The space of functions 
\begin{align*}
\mathcal{B}^{w}_\sigma=\{f_\nu:\nu\in w\cdot\mathcal{M}(\Theta)\}, \quad \text{with} \quad f_\nu(x)=\scal{\nu}{\sigma(\langle\cdot,x\rangle)}{w\cdot\mathcal{M}(\Theta)}{C_w(\Theta)},
\end{align*}
is a reproducing kernel Banach space if endowed with the norm
\[
\|f\|_{\mathcal{B}^{w}_\sigma}=\inf\{\|\nu\|_{w,*}:f=f_\nu\},\quad \|\nu\|_{w,*}=\|w^{-1}\cdot\nu\|_{\rm TV},
\]
and the associated ERM problem reads as follows
\begin{align*}
    \inf_{f \in \mathcal{B}^{w}_\sigma} \frac{1}{N} \sum_{i=1}^N L(y_i, f(x_i)) + \|f\|_{\mathcal{B}^{w}_\sigma}.
\end{align*}
Since the space $\mathcal{B}^{w}_\sigma$ is parametrized by the measure space $w\cdot\mathcal{M}(\Theta)$, the above learning task can be reformulated as a minimization problem over $w\cdot\mathcal{M}(\Theta)$
\begin{align}\label{eq:risk_w}
    \inf_{\nu \in w\cdot\mathcal{M}(\Theta)} \frac{1}{N} \sum_{i=1}^N L(y_i, f_\nu(x_i)) + \|\nu\|_{w,*}.
\end{align}
We can rephrase \eqref{eq:risk_w} using a more general setting and a $\KRu$ regularization may also be included. We write
\begin{align}\label{eq:risk_w_KR}
    \inf_{\nu \in w\cdot\mathcal{M}(Z)} F(\mathcal{A}\nu) + G^w_{\alpha,\beta}(\nu),
\end{align}
where $(Z,d)$ is a complete, separable, locally compact, geodesic and pointed metric space with base point $e \in Z$, the operator $ \mathcal{A} : w\cdot\mathcal{M}(Z) \to H$ is linear and continuous with values in the Hilbert space $H$, the functional
$ F : H \to (-\infty,+\infty] $
is proper, convex, coercive and lower semicontinuous and $G_{\alpha,\beta}^w \colon w\cdot \mathcal{M}(Z) \rightarrow [0,\infty]$ is given by
\begin{equation}
     G^w_{\alpha,\beta}(\nu) \defeq \alpha\|\nu\|_{\KRu}+\beta\|\nu\|_{w,*}.\end{equation}
In this setting we can directly establish existence of minimizers for problem~\eqref{eq:risk_w_KR}. In fact, $F\circ \mathcal{A} + G^w_{\alpha,\beta}$ is lower semicontinuous with respect to the weak$^*$ topology and, by the Banach-Alaoglu theorem, its sublevel sets are weak$^*$ compact.

For the choice of weight $w_q=(1+d(\cdot,e)^q)^{-1}$ with $q\in(0,+\infty)$, the norm reads
\[
\|f\|_{\mathcal{B}^{q}_\sigma}=\inf\{\|\nu\|_{{\mathcal{M}}_q}:f=f_\nu\},\quad\|\nu\|_{{\mathcal{M}}_q}=\int_\Theta (1+d(\theta,e)^q){\rm d}|\nu|(\theta),
\]
which corresponds to the regularizer used in  \cref{sec:regprobs}. As a consequence, by Lemma~\ref{lem:identification}, the minimization problem \eqref{eq:risk_w_KR} coincides with \eqref{eq:opt-prob} and we can directly establish existence of the minimizers when $p=q>1$, proving again Theorem~\ref{thm:existence_of_minimizers}. Note that whenever $p=q\leq 1$, we can still prove existence of minimizers in case $\alpha =0$. This setting, however, does not reflect the additional optimal transport effects induced by $\KRu$ compactness, which will be made explicit in the analysis of large data limits in Section~\ref{sec:samplinglimit}. 

Finally, we observe that this sheds light to an interesting connection with the paper \cite{bartolucci2023understanding}. In fact, the optimization problem \eqref{eq:risk_w} is equivalent to the TV-regularized ERM problem considered in \cite{bartolucci2023understanding}. Indeed, we can rewrite \eqref{eq:risk_w}
as
\begin{align}\label{eq:variational_problem_TV}
        \inf_{\nu \in w\cdot\mathcal{M}(\Theta)} \frac{1}{N} \sum_{i=1}^N L(y_i, f_\nu(x_i)) + \|\nu\|_{w,*} &= 
            \inf_{\mu \in\mathcal{M}(\Theta)} \frac{1}{N} \sum_{i=1}^N L(y_i, f_{w\cdot\mu}(x_i)) + \|\mu\|_{\TV},
       % &=\min_{\mu \in \M_1(Z)} F(\mathcal{A}\mu)+ G_{0,\beta}(\mu)\\
 %   &= \min_{\nu \in \M_1(Z)} F \left( \calA \frac{\nu}{1+d(\cdot,e)^p}\right)  + \beta\norm{\nu}_{\rm TV} \\
  %  &= \min_{\nu \in \M_1(Z)} \frac{1}{N} \sum_{i=1}^N L\left(y_i,\sp{\frac{\nu}{1+d(\cdot,e)^p}, \phi_i} \right)   + \beta\norm{\nu}_{\rm TV} \\
%    &= \min_{\nu \in \M_1(Z)} \frac{1}{N} \sum_{i=1}^N L\left(y_i,\sp{\nu, \frac{\phi_i(\cdot)}{1+d(\cdot,e)^p}} \right)   + \beta\norm{\nu}_{\rm TV},
\end{align}
where 
\begin{align}\label{eq:pairingg}
f_{w\cdot\mu}(x)=\scal{w\cdot\mu}{\sigma(\langle\cdot,x\rangle)}{w\cdot\mathcal{M}(\Theta)}{C_w(\Theta)}=\scal{\mu}{w^{-1}(\cdot)\sigma(\langle\cdot,x\rangle)}{\mathcal{M}(\Theta)}{C_0(\Theta)},
\end{align}
can be interpreted as a pairing between $C_0(\Theta)$ and $\M(\Theta)$, thus reducing to the framework of~\cite{bartolucci2023understanding}. This new perspective clarifies the role of the smoothing function $w^{-1}$, which was already introduced and recognized as technically crucial in \cite{bartolucci2023understanding}. Here, it is further proven to be necessary for ensuring the existence of non-trivial minimizers for the variational problem \eqref{eq:variational_problem_TV}.
 
\section{Representer theorems}\label{sec:representer}
In this Section~we aim at showing the validity of representer theorems for the abstract minimization problem \eqref{eq:opt-prob} and, as a consequence, for the $\KRu$-regularized risk minimization problem in Definition \ref{def:KRrisk}. To this end, as pointed out in \cite{bredies2020sparsity, boyer2019representer}, the first necessary step consists in characterizing the extremal points of the ball of the regularizer $G_{\alpha,\beta}$.

\subsection{Extremal points of the unit ball of \texorpdfstring{$G_{\alpha,\beta}$}{Ga,b}}

In what follows we aim at characterizing the extremal points of the unit ball
\begin{align*}
B_{\alpha,\beta} = \{\mu \in \mathcal{M}_p(Z) : G_{\alpha,\beta}(\mu) \leq 1\}
\end{align*}
for the case $\alpha >0$. If $\alpha = 0$ it is well-known that extremal points are Dirac deltas. We will discuss this case in Section~\ref{sec:classicalrepresenter}. 
Preliminarily, note that we can write $G_{\alpha,\beta}$ in terms of the balanced version of the KR norm (see Section~\ref{sec:opttrans}), that is 
\begin{align*}
G_{\alpha,\beta}(\mu)  = \alpha\|\mu - \mu(Z) \delta_e\|_{\rm KR} + \alpha|\mu(Z)| + \beta\int (1+ d(z,e)^p) \d|\mu|.
\end{align*}
To characterize  extremal points of $B_{\alpha,\beta}$ we will use an auxiliary functional in $\mathcal{M}_p^0(Z) \times \R$ (where $\mathcal{M}_p^0(Z)$ is the subspace of balanced measures in $\M_p(Z)$), defined as
\begin{align*}
\mathcal{G}_{\alpha,\beta}(\nu,c) = \alpha\|\nu\|_{\rm KR} + \alpha|c| +  \beta\int (1+ d(z,e)^p) \d|\nu + c\delta_e|.
\end{align*}
We also consider the following map
\begin{align}\label{eq:T}
S : \mathcal{M}_p(Z) \rightarrow \mathcal{M}_p^0(Z) \times \R, \qquad S(\mu) = (\mu - \mu(Z)\delta_e, \mu(Z)).
\end{align}
\begin{lemma}
The map $S$ defined in \eqref{eq:T} is bijective with inverse $S^{-1}(\nu,c) = \nu + c \delta_e$. Moreover, it is narrowly continuous and
\begin{align}\label{eq:equ}
G_{\alpha,\beta}(\mu) = \mathcal{G}_{\alpha,\beta}(S(\mu)) .
\end{align}
\end{lemma}
\begin{proof}
First, note that $S$ is surjective. Indeed, given $(\nu ,c) \in \mathcal{M}_p^0(Z) \times \R$ and defining $\mu = \nu + c\delta_e$,  since $\mu(Z) = c$, it holds that 
$S(\mu) = (\nu,c)$. Moreover, $S$ is also injective. Indeed, the equalities
\begin{align*}
\nu_1 + \nu_1(Z) \delta_e = \nu_2 + \nu_2(Z) \delta_e \quad \text{and} \quad \nu_1(Z) = \nu_2(Z) 
\end{align*}
imply immediately that $\nu_1 \mres (Z\setminus \{e\}) = \nu_2 \mres (Z\setminus \{e\})$. Additionally, since $\nu_1(Z) = \nu_2(Z)$ and $\nu_1 + \nu_1(Z) \delta_e = \nu_2 + \nu_2(Z) \delta_e$, it follows that $\nu_1 \mres \{e\} = \nu_2 \mres \{e\}$, implying that $\nu_1 = \nu_2$. From the bijectivity of $T$ it follows also that the inverse can be characterized as $S^{-1}(\nu,c) = \nu + c\delta_e$. Finally, narrow continuity is a consequence of the narrow continuity of the map $\mu \mapsto \mu(Z)$ and \eqref{eq:equ} can be readily verified.
\end{proof}
To characterize the extremal points of the unit ball of $\mathcal{G}_{\alpha,\beta}$, we will adapt an argument from the infimal convolution case treated in \cite{CarIglWal24p} (see \cite{aliaga2024convex} for related results).
\begin{lemma}\label{lem:extaux}
The extremal points of 
\begin{align*}
 \mathcal{B}_{\alpha,\beta} = \{(\nu,c) \in \mathcal{M}_p^0(Z) \times \R : \mathcal{G}_{\alpha,\beta}(\nu,c) \leq 1\}
\end{align*}
are necessarily of the form $(a\delta_x - a \delta_y, c)$ where $x,y \in Z$, $a > 0$ and $c\in \R$.
\end{lemma}
\begin{proof}
Consider $\nu \in \mathcal{M}^0_p(Z)$ such that $(\nu,c) \in {\rm Ext}(\mathcal{B}_{\alpha,\beta})$. 
Suppose that $\nu \neq 0$ is not a dipole. Since $\nu$ cannot be written as a dipole, there exist two disjoint measurable sets $E, F \subset Z$ such that $E \cup F = Z$ and $\nu^+(E)>0$ and $\nu^-(F)>0$. 
Denote by $\gamma_{opt}$ the optimal transport plan between $\nu^+$ and $\nu^-$.  Define now the following measures
\begin{align}
\nu_E = (\pi_2)_\# [\gamma_{opt} \mres (\pi_1^{-1}(E))], \qquad 
\nu_F = (\pi_2)_\# [\gamma_{opt} \mres (\pi_1^{-1}(F))].
\end{align}
Note that $\nu_E(Z) = (\nu^+ \mres  E)(Z)$ and $\nu_F(Z) = (\nu^+  \mres F)(Z)$. This follows from
\begin{align*}
\nu_E(Z) = (\pi_2)_\# [\gamma_{opt} \mres (\pi_1^{-1}(E))](Z) = \gamma_{opt} (\pi_1^{-1}(E)) = (\nu^+ \mres E)(Z)
\end{align*}
and a similar computation for $\nu_F$. Moreover, 
\begin{align}\label{eq:marginal}
\gamma_{opt}\mres (\pi_1^{-1}(E)) \in \Pi(\nu^+\mres E, \nu_E),\quad \ 
\gamma_{opt} \mres (\pi_1^{-1}(F)) \in \Pi(\nu^+ \mres  F, \nu_F).
\end{align}
Indeed for every measurable set $A \subset Z$ it holds that 
\begin{align*}
\gamma_{opt} \mres (\pi_1^{-1}(E))(A \times Z) = \gamma_{opt}((E \cap A) \times Z) = \nu^+\mres(E \cap A) = (\nu^+\mres E)(A) \\
\gamma_{opt}\mres (\pi_1^{-1}(E))(Z \times A) = \gamma_{opt} \mres (E\times A) = [\gamma_{opt} \mres (\pi^{-1}(E))](Z\times A) = \nu_E(A).
\end{align*}
Therefore, due to \eqref{eq:marginal}, \eqref{eq:otformula} and the definition of $\gamma_{opt}$ it holds
\begin{align*}
W_1(\nu^+ \mres E, \nu_E) + W_1(\nu^+ \mres  F, \nu_F) & \leq \int_{\pi_1^{-1}(E)} d(x,y) \d\gamma_{opt} + \int_{\pi_1^{-1}(F)} d(x,y) \d\gamma_{opt} \\
& = W_1(\nu^+, \nu^-) = \|\nu\|_{\rm KR}.
\end{align*}  
Define  the following constants
\begin{align*}
C_E =& W_1(\nu^+ \mres E, \nu_E) + |c| + \int_E (1+ d(z,e)^p) \d|\nu + c\delta_e|, \\
C_F =& W_1(\nu^+ \mres F, \nu_F) + \int_F (1+ d(z,e)^p) \d|\nu + c\delta_e|.
\end{align*}
Note that since $G_{\alpha,\beta}(\nu) \leq 1$, it holds that
\begin{align*}
C_E + C_F \leq \|\nu\|_{\rm KR} + c + \int_Z (1+ d(z,e)^p) \d|\nu + c\delta_e|  \leq 1,
\end{align*}
implying that $C_E, C_F \in (0,1)$. Now define the following measures
\begin{align*}
\nu_{1+} & = \frac{C_E + C_F}{C_E} \nu^+ \mres E, \quad \nu_{1-} = \frac{C_E + C_F}{C_E} \nu_E, \\
\nu_{2+} & = \frac{C_E + C_F}{C_F} \nu^+ \mres F, \quad \nu_{2-} = \frac{C_E + C_F}{C_F} \nu_F,
\end{align*}
and let $c_1 = c\frac{C_E + C_F}{C_E}$. First note that $\frac{C_E}{C_E + C_F} \nu_{1+} + \frac{C_F}{C_E + C_F}\nu_{2+} = \nu^+$ and $\frac{C_E}{C_E + C_F}\nu_{1-} + \frac{C_F}{C_E + C_F}\nu_{2-} = \nu^-$. While the first identity is clear, for the second one we can note that for every measurable set $A \subset Z$ it holds that
\begin{align*}
\frac{C_E}{C_E + C_F}(\nu_{1-} + \nu_{2-})(A) & = \nu_E(A) + \nu_F(A) = (\gamma_{opt}(E \times A) + \gamma_{opt}(F \times A)) \\
& = (\gamma_{opt}(Z \times A)) =  \nu^-(A).
\end{align*}
Therefore
\begin{align*}
(\nu,c) = \frac{C_E}{C_E + C_F} ((\nu_{1+},c_1) - (\nu_{1-},0)) + \frac{C_F}{C_E + C_F} ((\nu_{2+},0) - (\nu_{2-},0)). 
\end{align*}
To achieve a contradiction it remains to verify that $\mathcal{G}_{\alpha,\beta}((\nu_{1+},c_1) - (\nu_{1-},0)) \leq 1$ and $\mathcal{G}_{\alpha,\beta}((\nu_{2+},0) - (\nu_{2-},0)) \leq 1$. We only verify the first inequality:
\begin{align*}
 &\mathcal{G}_{\alpha,\beta}((\nu_{1+},c_1) - (\nu_{1-},0)) \\ 
 = & \frac{C_E + C_F}{C_E} (W_1(\nu^+ \mres E,\nu_E) + |c| + \int_E (1+ d(z,e)^p) \d|\nu^+ \mres E - \nu_{E} + c\delta_e|)\\
 \leq & \frac{C_E + C_F}{C_E}(W_1(\nu^+ \mres E,\nu_E) + |c| + \int_E (1+ d(z,e)^p) \d|\nu + c\delta_e|) \leq 1.
\end{align*}
\end{proof}

We are finally ready to characterize the extremal points of $B_{\alpha,\beta}$.
\begin{theorem}\label{thm:extremal}
The extremal points of $B_{\alpha,\beta}$ are necessarily of the following form: 
\begin{align}\label{eq:dipole}
\mu = a\delta_x - a \delta_y + \mu(Z)\delta_e
\end{align}
where $x,y \in Z$ and $a > 0$ is a normalizing constant satisfying
\begin{align}\label{eq:valuea}
    1 = \alpha (a d(x,y) + |\mu(Z)|) + \beta \int_Z (1 + d(z,e)^p)\, \d|a\delta_x - a \delta_y + \mu(Z)\delta_e|(z).
\end{align}
\end{theorem}

\begin{proof}
First, observe that $B_{\alpha,\beta} = S^{-1}(\mathcal{B}_{\alpha,\beta})$.
Indeed, for $\mu$ such that $G_{\alpha,\beta}(\mu) \leq 1$, from \eqref{eq:equ} it holds that $\mathcal{G}_{\alpha,\beta}(S(\mu)) \leq 1$, that is $S(\mu) \in B_{\alpha,\beta}$ and thus $\mu \in S^{-1} (\mathcal{B}_{\alpha,\beta})$. Vice versa, if $\mu \in S^{-1}(\mathcal{B}_{\alpha,\beta})$, then there exists $(\nu,c)$ such that $\mathcal{G}_{\alpha,\beta}(\nu,c) \leq 1$ and $S(\mu) = (\nu,c)$. 
Therefore, from \eqref{eq:equ} we conclude that $G_{\alpha,\beta}(\mu) = \mathcal{G}_{\alpha,\beta}(S(\mu)) \leq 1$. From \cite[Lemma~3.2]{bredies2020sparsity} we have that ${\rm Ext} (B_{\alpha,\beta}) = S^{-1}{\rm Ext}(\mathcal{B}_{\alpha,\beta})$.
Therefore,~\eqref{eq:dipole} follows from Lemma~\ref{lem:extaux} and the definition of $S^{-1}$. To determine the value of $a$ we  note that, due to  one-homogeneity of $G_{\alpha,\beta}$, for every $\mu \in {\rm Ext} (B_{\alpha,\beta})$ one has  either $\mu = 0$ or $G_{\alpha,\beta}(\mu) = 1$. By plugging \eqref{eq:dipole} into the equation $G_{\alpha,\beta}(\mu) = 1$, we obtain~\eqref{eq:valuea}.
\end{proof}

\subsection{Representer theorems for \texorpdfstring{$\KRu$}{KRu}-regularized empirical risk minimization}

The knowledge of the extremal points of the unit ball of $G_{\alpha,\beta}$ allows to prove representer theorems for the abstract minimization problem in \eqref{eq:opt-prob}.

\begin{theorem}\label{thm:rep}
Suppose that $H = \R^N$, $p>1$ and $\alpha >0$. Then there exists a minimizer $\mu_{\min} \in \mathcal{M}_p(Z)$ 
of \eqref{eq:opt-prob} that can be written as
\begin{align}\label{eq:representation}
    \mu_{\min} = \sum_{i=1}^N c_i (\delta_{z_i} - \delta_{w_i}) + \mu_{\min}(Z)\delta_e \quad \text{where }c_i \in \R \text{ and }z_i,w_i \in Z.
\end{align}
\end{theorem}
\begin{proof}
    The proof follows from a direct application of  \cite{bredies2020sparsity}. Indeed, $G_{\alpha,\beta}$ is a seminorm with a trivial nullspace. Due to Theorem~\ref{thm:compact-embedding-locally-compact}, its sublevel sets are compact in the topology induced by the norm $\KRu$ and due to Theorem~\ref{thm:lsc}, $G_{\alpha,\beta}$ is lower semicontinuous with respect to the {\rm KRu} norm. Therefore, by applying \cite[Theorem~3.3]{bredies2020sparsity} together with Theorem~\ref{thm:extremal} we immediately obtain the representation in \eqref{eq:representation}.
\end{proof}
In particular, the previous theorem can be specialized to the case of  $\KRu$-regularized ERM problem introduced in Definition \ref{def:KRrisk} in case $\alpha >0$, $p>1$ and for measures defined on a set of weights $\Theta \subset \R^{d+1}$. Indeed, since the linear operator $\mathcal{A}$ defined in \eqref{eq:A} is mapping to $\R^N$ and $F$ defined in \eqref{eq:F} is convex, coercive and lower semicontinuous, we immediately deduce from Theorem~\ref{thm:rep} the existence of a measure $\mu_{\min} \in \M_p(\Theta)$ minimizer of \eqref{eq:risk_new} that can be written as in \eqref{eq:representation}.

\subsubsection{Recovering the representer theorem of
\texorpdfstring{\cite{bartolucci2023understanding}}{B-DV-R-V} for \texorpdfstring{$\alpha = 0$}{a=0}}\label{sec:classicalrepresenter}
Note that the representer theorem in Theorem~\ref{thm:rep} is valid under the condition that $\alpha >0$. In case $\alpha = 0$, the minimization problem \eqref{eq:opt-prob} reduces to a variant of a TV-regularized risk minimization problem. In this case a simpler representer theorem, that is a variant of the one obtained in \cite{bartolucci2023understanding} can be derived thanks to the following observation. 

\begin{lemma}\label{thm:ext-1}
For $p\geq 1$, the extremal points of the unit ball of $G_{0,1}$ are given by
    \begin{equation*}
        \Ext(B_{0,1}) = \left\{ \frac1{1+d(z,e)^p} \delta_z, \; z \in Z \right\}.
    \end{equation*}
\end{lemma}
\begin{proof}
    For a measure $\mu \in \M_p(Z)$ define the map $T : \M_p(Z) \rightarrow \M(Z)$ as
    \begin{equation}\label{eq:nu-mu}
        T(\mu)(A) \defeq \int_A ( 1 + d(z,e)^p ) \, \d \mu(z) \quad \text{for any Borel subset $A \subset Z$.}
    \end{equation}
     Note that $T$ is bijective and $G_{0,1}(\mu) = \norm{T(\mu)}_{\rm TV}$. Since we have that for this norm $\Ext(\{ \mu \in \M(Z) : \|\mu\|_{\rm TV} \leq 1) = \left\{ \delta_z, \; z \in Z \right\}$, the claim follows from \cite[Lemma~3.2]{bredies2020sparsity}.
\end{proof}

\begin{theorem}\label{thm:rep-01}
Suppose that $H=\R^N$, $p>1$, $\alpha = 0$. Then there exists a minimizer $\mu_{\min} \in \mathcal{M}_p(Z)$ 
of \eqref{eq:opt-prob} that can be represented as
\begin{align}
    \mu_{\min} = \sum_{i=1}^N \frac{c_i}{1 + d(z_i,e)^p} \delta_{z_i} \quad \text{ where }c_i \in \R \text{ and }z_i \in Z.
\end{align}
\end{theorem}
\begin{proof}
    The proof follows again by applying \cite[Theorem~3.3]{bredies2020sparsity} using the characterization of extremal points of the unit ball of $G_{0,1}$ obtained in Theorem~\ref{thm:ext-1}. 
\end{proof}
This representer theorem can be specialized to the case of the TV-regularized ERM problem for measures defined on a set of weights $\Theta \subset \R^{d+1}$, i.e. \eqref{eq:risk_new} for $\alpha = 0$, $\beta >0$ and $p>1$.
It is worth noting that the representer theorem \cite[Theorem~3.9]{bartolucci2023understanding} can be recovered from \cref{thm:rep-01} simply by observing that for $w_p=(1+d(\cdot,e)^p)^{-1}$
\begin{align}
    & \inf_{\mu \in \M_p(\Theta)} \frac{1}{N} \sum_{i=1}^N L(y_i, f_\mu(x_i))  +  \beta\int_\Theta (1 + d(\theta,e)^p) \d \abs{\mu}(\theta) \nonumber\\
    & = \inf_{\nu \in \M(\Theta)} \frac{1}{N} \sum_{i=1}^N L(y_i, f_{w_p\cdot \nu}(x_i))    + \beta \norm{\nu}_{\rm TV} \label{eq:TVri}
\end{align}
where $f_{w_p\cdot \nu}$ is defined as in \eqref{eq:pairingg} and it can be interpreted as a pairing between $C_0(\Theta)$ and $\mathcal{M}(\Theta)$ (see Section~\ref{sec:controlledgrowth}), thus recovering the risk minimization problem analyzed in~\cite{bartolucci2023understanding}. In particular, by noticing that minimizers of \eqref{eq:TVri} can be constructed from minimizers of \eqref{eq:risk_new} (for $\alpha = 0$) via the map $T$ in \eqref{eq:nu-mu}, it is easy to see that Theorem~\ref{thm:rep-01} implies the representer theorem \cite[Theorem~3.9]{bartolucci2023understanding}.

\subsection{A representer theorem for distillation problems}
In this Section~we prove a representer theorem for the distillation of an infinitely wide neural network presented in Section~\ref{sec:applications}.   
To this aim, we consider a reference measure $\mu^* \in \mathcal{M}_p(\Theta)$ defined on a set of weights $\Theta \subset \R^{d+1}$. The representer theorem is stated as follows.

\begin{theorem}\label{thm:rep_distillation}
   Suppose that $L(y,w) = \ell(y - w)$, where $\ell: \R \rightarrow \R$ is a convex, coercive and lower semicontinuous loss. If $p>1$ there exists a solution of \eqref{eq:risk_distillation} that can be written as 
    \begin{align}\label{eq:repdi}
    \mu_{\min} = \sum_{i=1}^N c_i (\delta_{\theta_i} -  \delta_{\eta_i}) + \mu_{\min}(\Theta)\delta_e + \mu^* \quad \text{ for }\theta_i,\eta_i \in \Theta \text{ and }c_i \in \R.   
    \end{align}
\end{theorem}
\begin{proof}
Consider the following two functionals
\begin{align*}
    &G_1(\mu) = \frac{1}{N} \sum_{i=1}^N \ell(y_i - f_\mu(x_i)) +  G_{\alpha,\beta}(\mu - \mu^*),   \\
    &G_2(\mu) = \frac{1}{N} \sum_{i=1}^N \ell(y_i - f_{\mu + \mu^*}(x_i)) + G_{\alpha,\beta}(\mu).
\end{align*}
They have the same infimum in $\KRu(\Theta)$, and if $\bar \mu$ is a minimizer of $G_2$, then $\bar \mu + \mu^*$ is a minimizer of $G_1$. Moreover, noting that $f_{\mu + \mu^*}(x_i) = f_{\mu^*}(x_i) + f_{\mu}(x_i)$ we can write $G_2$ as in \eqref{eq:opt-prob} by choosing $\mathcal{A} : \KRu(\Theta) \rightarrow \R^N$ as $(\mathcal{A}\mu)_i = \scal{\mu}{\sigma(\langle \cdot , x_i\rangle)}{\KRu(Z)}{\Lip(\Theta)}$ and 
$F(w_1,\ldots,w_N) = \frac{1}{N}\sum_{i=1}^N \ell(y_i - f_{\mu^*}(x_i) - w_i)$. 
Therefore, by applying Theorem~\ref{thm:rep} to $G_2$ we deduce the existence of a minimizer of $G_2$ of the form \eqref{eq:representation} and consequently a minimizer of $G_1$ written as \eqref{eq:repdi}. \sloppy
\end{proof}

\subsection{The non-compact case}

If $p=1$, we have seen that existence of minimizers for \eqref{eq:opt-prob} is not ensured. 
However, despite the bad-behaved nature of the problem we want to provide a representer theorem in case minimizers exist. This is achieved by using the general representer theorem proven in \cite{boyer2019representer} that does not require topological assumptions to hold.
For the sake of clarity we introduce some necessary definitions from convex analysis. For details, we refer the interested reader to \cite{boyer2019representer}.

\begin{definition}\label{def:convexanalysis}
Let $V$ be a vector space and $A \subset V$:
\begin{itemize}
    \item[i)] We say that $A$ is linearly closed if the intersection~of $A$ with every line $\{\mu + tv :\, t \in \R\}$ with $\mu \in  V$ and $v \in V \setminus \{0\}$ is closed with respect to the finite dimensional topology of the line.
\item[ii)] The recession cone of $A$, denoted as $\text{rec}(A)$, is defined as the collection of all $v \in V$ such that $A + \R_+v$ is contained in $A$. The lineality space of $A$, denoted as $\text{lin}(A)$ is defined as
$\text{lin}(A) := \text{rec}(A) \cup (- \text{rec}(A))$.
\item[iii)] A ray of $A$ is any half-line $\{\mu + tv :\, t \in \R_+\} \subset A$ with $\mu \in  V$ and $v \in V \setminus \{0\}$.
% that is contained in $A$. 
\end{itemize}
\end{definition}

\begin{theorem}\label{thm:repp=1}
Suppose that $H=\R^N$ and $p=1$. Moreover, assume that the set of minimizers of \eqref{eq:opt-prob} is non-empty and weak$^*$ compact. Then, at least one of these minimizers is represented as 
\begin{align*}
   \mu_{\min} = \sum_{i=1}^N c_i (\delta_{z_i} - \delta_{w_i}) + \mu_{\min}(Z)\delta_e, \quad \text{where }c_i \in \R \text{ and }z_i,w_i \in Z. 
\end{align*}
\end{theorem}
\begin{proof}
We apply Corollary 2 in \cite{boyer2019representer} to the functional $\eqref{eq:opt-prob}$ for $\mu \in \mathcal{M}_1(Z)$. First we verify that the $s$-sublevel sets of $G_{\alpha,\beta}$ denoted by $\mathcal{S}(G_{\alpha,\beta},s)$ are linearly closed according to Definition \ref{def:convexanalysis}. Consider a line in $\mathcal{M}_1(Z)$, that is $\{\mu + tv :\, t \in \R\}$ where $\mu \in  \mathcal{M}_1(Z)$, $v \in \mathcal{M}_1(Z) \setminus \{0\}$. Consider now a sequence of $t_n$ such that $t_n \rightarrow t$ and $\mu + t_n v \in \mathcal{S}(G_{\alpha,\beta},s)$. To prove that $\mu + t v \in \mathcal{S}(G_{\alpha,\beta},s)$ it is enough to note that 
\begin{align*}
   &\lim_{n\rightarrow+\infty} |G_{\alpha,\beta}(\mu + t_n v) -  G_{\alpha,\beta}(\mu + t v)| \\
   \leq&  \lim_{n\rightarrow+\infty} |t_n - t| \Bigg( \alpha \|v\|_{\KRu} + \beta \int_Z (1 + d(z,e)) \d|v| \Bigg) = 0.
\end{align*}
Now, let $v \in \mathcal{M}_1(Z)$ be such that $v \neq 0$ and consider a sequence $t_n \rightarrow \infty$ and $\mu \in \mathcal{S}(G_{\alpha,\beta},s)$. Then 
\begin{align*}
    \lim_{n\rightarrow +\infty} G_{\alpha,\beta}(\mu + t_n v) & \geq  \lim_{n\rightarrow +\infty} \int_Z (1 + d(z,e))\, \d|\mu + t_n v| \\
    & \geq \lim_{n\rightarrow +\infty}  t_n \int_Z (1 + d(z,e))\d|v| - \int_Z (1 + d(z,e))\d|\mu| = +\infty
\end{align*}
implying that $v \notin \text{rec}(\mathcal{S}(G_{\alpha,\beta},s))$ and thus $\text{lin}(\mathcal{S}(G_{\alpha,\beta},s)) =\{0\}$. Moreover, $\mathcal{S}(G_{\alpha,\beta},s)$ contains no rays.
Note now that $\inf_{\mathcal{M}_1(Z)} G_{\alpha,\beta}(\mu) = 0$, that is achieved only at $\mu = 0$. Therefore, without loss of generality we can assume that
\begin{align*}
\inf_{\mu \in \mathcal{M}_1(Z)} G_{\alpha,\beta}(\mu) < G_{\alpha,\beta}(\mu^*)
\end{align*}
where $\mu^*$ is a minimizer of \eqref{eq:opt-prob}.
We can then apply Corollary 2 in \cite{boyer2019representer}, to say that, since $\mathcal{S}(G_{\alpha,\beta},s)$ does not contain any rays, all extremal points of the solution set can be written as a linear combination of $N$ extremal points ($0$-dimensional faces) of the ball of $G_{\alpha,\beta}$. Note that since the set of minimizers is non-empty and weak$^*$ compact, 
the set of its extremal points is non-empty due to Krein-Milman theorem. Therefore, by applying Lemma~\ref{thm:ext-1} with $p=1$ we conclude.
\end{proof}

Similarly to the  previous sections, Theorem~\ref{thm:repp=1} can be specialized to the case of  $\KRu$-regularized ERM problem introduced in Definition \ref{def:KRrisk} in case $\alpha >0$, $p=1$.

%%%%%%%%%%%%%%%%%%%%%%%%%%%%%%%%%%%%%%%%%
\section{Strong \texorpdfstring{$\KRu$}{KRu} convergence of minimizers in the infinite data limit}\label{sec:samplinglimit}

In this section, we explore some consequences of the compactness in KRu norm given by \cref{thm:compact-embedding-locally-compact} for the learning of representations arising from the mean field parametrization \eqref{eq:infinite_width_nn}. In particular we consider a measure $\mu \in \mathcal{M}_1(\Theta)$ on a set of weights $\Theta \subset \R^{d+1}$ and a parametrization $f_\mu:X \to \R$ as in \eqref{eq:infinite_width_nn}, where $X \subseteq \R^d \times \{1\}$.
% is the dataset.

We start with a direct implication of strong convergence of parameter measures $\mu \in \mathcal{M}_1(\Theta)$ in the KRu norm:
\begin{lemma}\label{lem:unifconv}Denote by $f_\mu:X \to \R$ for $\mu \in \M_1(\Theta)$ the parametrization given by \eqref{eq:infinite_width_nn}. If a sequence $\mu_n \in \M_1(\Theta)$ converges strongly in KRu norm to some $\mu_\infty$, that is $\norm{\mu_n - \mu_\infty}_\KRu \to 0$, we have that the functions $f_{\mu_n}$ converge to $f_{\mu_\infty}$ uniformly on bounded sets of $X$.    
\end{lemma}
\begin{proof}This follows directly from the duality of \cref{thm:dualityKR} and the expression \eqref{eq:meanfield}. Namely, for any $R>0$ we have
    \begin{equation*}
        \sup_{x \colon \norm{x} \leq R} \abs{f_{\mu_N^\beta}(x) - f_{\mu_\infty}(x)} = \sup_{x \colon \norm{x} \leq R} \abs{\int_\Theta \sigma(\sp{\omega,x}) \d(\mu_N^\beta - \mu_\infty)} \to 0 
    \end{equation*}
    since the functions $\omega \mapsto \sigma(\sp{\omega,x})$ are $\norm{x}$-Lipschitz. 
\end{proof}
For the rest of this section, we consider the following regularized regression problem with $N$ data points:
\begin{equation}\label{eq:minimisers_N_beta}
    \mu_N^\beta \in \argmin_{\mu \in \KRu(\Theta)} \frac{1}{N} \sum_{i=1}^N L(y_i, f_\mu(x_i)) + \beta \int_\Theta ( 1 + d(\theta,e)^p ) \, \d\abs{\mu}(\theta),
\end{equation}
where $(x_i,y_i) \iid \rho$ are samples from some joint probability measure $\rho \in \calP(X \times \R)$. One could also consider the regularizer $G_{\alpha,\beta}$ defined in~\eqref{eq:G-definition}, but we restrict ourselves to the case $\alpha=0$ for simplicity. 
Suppose that there is no noise in the data $\{(x_i,y_i)\}_{i=1}^N$ and the model is well specified, i.e. there exists a $\mu$ such that $y=f_\mu(x)$  $\rho$-a.e. We want to examine the behavior of minimizers $\mu_N^\beta$ as $N \to \infty$ and $\beta \to 0$. 
\begin{theorem}\label{thm:minimisers-strong-conv-KR}
    Suppose that the assumptions above hold, and that the loss $L \colon \R \times \R \to \R_+$ is continuous and satisfies the following conditions for some $C_1,C_2,r>0$: 
    \begin{subequations}
    \begin{alignat}{2}
        L(y,y' + z-z') &\leq C_1(L(y,y') + L(z,z'))\quad &&\text{ for any }\ y,y',z,z' \in \R;\label{eq:triangle}\\
        L(y,y') &\leq C_2 \abs{y-y'}^r&&\text{ for any }\ y,y' \in \R;\label{eq:LrLip}\\
        L(y,y') & = 0 &&\text{ only if} \quad y = y'. \label{eq:zero}
    \end{alignat}
    \end{subequations}
    Suppose also that the function $f_\mu(x)$ defined in~\eqref{eq:infinite_width_nn} is uniformly continuous, positively one-homogeneous, that is $f_\mu(tx) = t f_\mu(x)$ for any $x \in X$ and $t>0$, and that the probability measure $\rho$ has $\tilde{r} > r$ finite moments, that is
    \begin{equation*}
        \int_{X \times \R} \norm{x}^{\tilde{r}} + |y|^{\tilde{r}} \d\rho(x,y) < \infty.
    \end{equation*}
    Then with probability $1$, the sampled data $\{(x_i,y_i)\}_{i \in \N}$ are such that for any sequence of minimizers as in \eqref{eq:minimisers_N_beta}, there is a subsequence that converges as $N \to \infty$ and $\beta \to 0$ to a solution
    \begin{equation}\label{eq:mu-dagger-argmin}
    \mu^\dagger \in \argmin_{\mu \in \KRu(\Theta)} \left\{\int_\Theta ( 1 + d(\theta,e)^p ) \, \d\abs{\mu}, \quad \text{ s.t. $y=f_\mu(x)$  $\rho$-a.e.}\right\}.
\end{equation}
\end{theorem}
\begin{proof}
    Let $\mu^* \in \KRu(\Theta)$ be such that $y=f_{\mu^*}(x)$  $\rho$-a.e. (which exists because we assumed a well-specified model) but otherwise arbitrary. Comparing the value of the objective function~\eqref{eq:minimisers_N_beta} at $\mu^*$ and the minimizer $\mu_N^\beta$, with probability $1$ over the sequence $\{(x_i,y_i)\}_{i \in \N}$ (in the countably infinite product probability space that makes all of them iid random variables, which is unique, see \cite[Prop.~10.6.1]{Coh13}) we get
    \begin{align*}
        & \frac{1}{N} \sum_{i=1}^N L(y_i, f_{\mu_N^\beta}(x_i)) + \beta \int_\Theta ( 1 + d(\theta,e)^p ) \, \d\abs{\mu_N^\beta} \\
        \leq & \frac{1}{N} \sum_{i=1}^N L(y_i, f_{\mu^*}(x_i)) + \beta\int_\Theta ( 1 + d(\theta,e)^p ) \, \d\abs{\mu^*} 
        =  \beta\int_\Theta ( 1 + d(\theta,e)^p ) \, \d\abs{\mu^*}.
    \end{align*}
    Therefore, $\int_\Theta ( 1 + d(\theta,e)^p ) \, \d\abs{\mu_N^\beta} \leq \int_\Theta ( 1 + d(\theta,e)^p ) \, \d\abs{\mu^*}$  uniformly in $N,\beta$ and, by Theorem~\ref{thm:compact}, we get, with probability $1$, a converging subsequence (which we do not relabel) such that $\norm{\mu_N^\beta - \mu^\dagger}_\KRu \to 0$. 
    In particular, as in \cref{lem:unifconv}, this implies
    \begin{equation}\label{eq:conv-unif}
        \sup_{x \colon \norm{x} \leq 1} \abs{f_{\mu_N^\beta}(x) - f_{\mu^\dagger}(x)} = \sup_{x \colon \norm{x} \leq 1} \abs{\int_\Theta \sigma(\sp{\omega,x}) \d(\mu_N^\beta - \mu^\dagger)} \to 0.
    \end{equation}
    Now we want to show that $y=f_{\mu^\dagger}(x)$  $\rho$-a.e. Using \eqref{eq:triangle} we have that
    \begin{align*}
        \frac{1}{N} \sum_{i=1}^N L(y_i, f_{\mu^\dagger}(x_i)) &\leq \frac{1}{N} \sum_{i=1}^N L(y_i, f_{\mu_N^\beta}(x_i) + f_{\mu^\dagger}(x_i) - f_{\mu_N^\beta}(x_i) ) \\
        &\leq \frac{1}{N} \sum_{i=1}^N C_1\left[ L(y_i, f_{\mu_N^\beta}(x_i)) + L(f_{\mu^\dagger}(x_i), f_{\mu_N^\beta}(x_i) ) \right].
    \end{align*}
    The first summand is upper bounded by $\beta\int_\Theta ( 1 + d(\theta,e)^p ) \, \d\abs{\mu^*}$. For the second one, using \eqref{eq:LrLip}, we obtain
    \begin{align*}
        & \frac{1}{N} \sum_{i=1}^N L(f_{\mu^\dagger}(x_i), f_{\mu_N^\beta}(x_i) ) \leq \frac{1}{N} \sum_{i=1}^N C_2\abs{f_{\mu^\dagger}(x_i) - f_{\mu_N^\beta}(x_i)}^r \\
        =& \frac{1}{N} \sum_{i=1}^N C_2\norm{x_i}^r \abs{f_{\mu^\dagger}\left(\frac{x_i}{\norm{x_i}}\right) - f_{\mu_N^\beta}\left(\frac{x_i}{\norm{x_i}}\right)}^r \\
        \leq& \left(\frac{1}{N} \sum_{i=1}^N C_2\norm{x_i}^r\right) \sup_{j=1,...,N} \abs{f_{\mu^\dagger}\left(\frac{x_j}{\norm{x_j}}\right) - f_{\mu_N^\beta}\left(\frac{x_j}{\norm{x_j}}\right)}^r.
    \end{align*}
    Hence, we get the estimate (dropping both constants $C_1$ and $C_2$ for readability)
    \begin{align}
        \frac{1}{N} \sum_{i=1}^N L(y_i, f_{\mu^\dagger}(x_i)) &\lesssim \beta\int_\Theta ( 1 + d(\theta,e)^p ) \, \d\abs{\mu^*} \nonumber \\ &\qquad+ \left(\frac{1}{N} \sum_{i=1}^N \norm{x_i}^r\right) \sup_{j=1,...,N} \abs{f_{\mu^\dagger}\left(\frac{x_j}{\norm{x_j}}\right) - f_{\mu_N^\beta}\left(\frac{x_j}{\norm{x_j}}\right)}^r. \label{eq:sup}
    \end{align}
    Note that the first term on the right-hand side converges to zero as $\beta \to 0$. Moreover, by the general version for separable metric spaces (\cite{Var58}, see also \cite[Thm.~11.4.1]{Dud02}) of the Glivenko-Cantelli theorem we have, again with probability $1$ over the sequence $\{(x_i,y_i)\}_{i \in \N}$, that the empirical measure $\rho_N := \frac{1}{N} \sum_{i=1}^N \delta_{(x_i,y_i)}$ 
    converges narrowly to the original distribution $\rho$, that is
    \begin{align*}
    \lim_{N \to \infty} \int_{X \times \R} g(x,y) \d\rho_N(x,y) = \int_{X \times \R} g(x,y) \d\rho(x,y) \ \ \text{for all }g \in C_b(X \times \R).
    \end{align*}
    Now, since we have assumed that $\rho$ has $\tilde{r}>1$ finite moments we can use Banach-Alaoglu (upon possibly selecting a further non relabeled subsequence) in the space $\mathcal{M}_{\bar r}(X \times \R)$ defined in \eqref{eq:def_of_Mp} to infer that
    \begin{align}\label{eq:convv}
    \lim_{N \to \infty} \int_{X \times \R} g(x,y) \d\rho_N(x,y) =  \int_{X \times \R} g(x,y) \d\rho(x,y) \ \ \text{for all }g \in C_{w_{\bar r}}(X \times \R)
    \end{align}
    where $w_{\bar r}(x,y) = (1+\|x\|^{\bar r}+|y|^{\bar r})^{-1}$ for $r < \bar r < \tilde r$ (see \eqref{eq:dualnorm} and \eqref{eq:momentsweights}). Hence,  using as test function $g(x,y) = \|x\|^r \in C_{w_{\bar r}}(X \times \R)$, we obtain that the term $\left(\frac{1}{N} \sum_{i=1}^N \norm{x_i}^r\right)$ is bounded uniformly in $N$ because $\rho$ has $r$ finite moments.
    Therefore, since the supremum term in \eqref{eq:sup} converges with probability $1$ to zero due to~\eqref{eq:conv-unif}, we get that $\frac{1}{N} \sum_{i=1}^N L(y_i, f_{\mu^\dagger}(x_i)) \to 0$ with probability $1$ as $N \to \infty$ and $\beta \to 0$.

    Since $f_{\mu^\dagger}$ is positively homogeneous and uniformly continuous, it is also Lipschitz. Therefore, since and $L$ satisfies \eqref{eq:LrLip}, we have that $g(x,y) = L(y, f_{\mu^\dagger}(x)) \in C_{w_{\bar r}}(X \times \R)$. Applying \eqref{eq:convv} with such $g$, we conclude that
    \begin{align*}
        0 = \lim_{N \to \infty} \frac{1}{N} \sum_{i=1}^N L(y_i, f_{\mu^\dagger}(x_i)) = \int_{X \times \R} L(y,f_{\mu^\dagger}(x)) \d\rho(x,y),
    \end{align*}
    so that $y=f_{\mu^\dagger}(x)$ $\rho$-a.e., thanks to \eqref{eq:zero}. By \cref{thm:lsc} we get that
    \begin{align*}
        \int_\Theta ( 1 + d(\theta,e)^p ) \, \d\abs{\mu^\dagger} \leq \int_\Theta ( 1 + d(\theta,e)^p ) \, \d\abs{\mu^*}.
    \end{align*}
    Since $\mu^*$ was arbitrary, this implies~\eqref{eq:mu-dagger-argmin}.
\end{proof}

\begin{example}
    A common example of a loss function satisfying the above conditions is $L(y,y') = \abs{y-y'}^r$ for $r\geq1$. Indeed, since the function $\abs{\cdot}^r$ is convex and absolutely $r$-homogeneous, we have that $\abs{y-y'}^r \leq 2^{r-1} \left( \abs{y}^r + \abs{y'}^r \right)$, and the first inequality follows. The second one is trivial.
\end{example}

\begin{example}
    A similar argument can be applied in the setting of inverse learning theory~\cite{bissantz2004consistency, blanchard2018optimal, bubba2023convex}. Here the goal is to reconstruct an unknown function $f$ from evaluations of \emph{another} function $g$, related to the first one via the action of a forward operator $K$: $g=Kf$. 
    Suppose that $X \subset \R^d$ is compact and $K \colon C(X) \to C(\Omega)$ is a linear bounded operator mapping continuous functions on $X$ to continuous functions on some compact $\Omega \subset \R^s$. The learning problem~\eqref{eq:minimisers_N_beta} becomes     \begin{equation}\label{eq:minimisers_N_beta-inv-prob}
    \mu_N^\beta \in \argmin_{\mu \in \KRu(\Theta)} \frac{1}{N} \sum_{i=1}^N L(g_i, [Kf_\mu](t_i)) + \beta \int_\Theta ( 1 + d(\theta,e)^p ) \, \d\abs{\mu}(\theta),
    \end{equation}
    where $(t_i,g_i) \iid \rho$ are samples from a joint probability measure $\rho \in \calP(\Omega \times \R)$. With similar arguments as in \cref{thm:minimisers-strong-conv-KR}, we get a convergent subsequence such that $\mu_N^\beta \to \mu^\dagger$ (strong convergence in $\KRu$) and, subsequently, $f_{\mu_N^\beta} \to f_{\mu^\dagger}$ strongly in $C(X)$ since $X$ is compact. Using \eqref{eq:LrLip}, we obtain
    \begin{align*}
        & \frac{1}{N} \sum_{i=1}^N L([Kf_{\mu^\dagger}](t_i), [Kf_{\mu_N^\beta}](t_i) ) \leq \frac{1}{N} \sum_{i=1}^N C_2\abs{[Kf_{\mu^\dagger}](t_i) - [Kf_{\mu_N^\beta}](t_i)}^r \\
        \leq & C_2\norm{Kf_{\mu^\dagger} - Kf_{\mu_N^\beta}}_{C(\Omega)}^r 
        \leq C_2\norm{K}^r\norm{f_{\mu^\dagger} - f_{\mu_N^\beta}}_{C(X)}^r \to 0
    \end{align*}
    with probability $1$. Proceeding as in \cref{thm:minimisers-strong-conv-KR}, we conclude that the minimizers converge strongly in $\KRu(\Theta)$ to some 
    \begin{equation*}
        \mu^\dagger \in \argmin_{\mu \in \KRu(\Theta)} \left\{\int_\Theta ( 1 + d(\theta,e)^p ) \, \d\abs{\mu}, \quad \text{ s.t. $g=[Kf_\mu](t)$  $\rho$-a.e.}\right\}
    \end{equation*}
     and the corresponding functions $f_{\mu_N^\beta}$ converge to $f_{\mu^\dagger}$ uniformly on $X$.
\end{example}

\section{Numerical implementation of distillation}\label{sec:experiments}
In this section we implement numerically a simple, low-dimensional example of neural network distillation discussed in Section~\ref{sec:applications}. Let
\begin{align*}
    \mu^* = \sum_{i=1}^{M_{T}} a_i \delta_{(w_i,b_i)}  
\end{align*}
be a given teacher network. Consider the optimization problem \cref{eq:risk_distillation}. By \cref{thm:rep_distillation} we know we can restrict ourselves to linear combinations of extremal points of the $\KR$ norm with a number of dipoles equal to $M\leq N$ (where $N$ is the number of data points), that is 
\begin{align*}
    \mu_{S} = \sum_{i=1}^{M} c_i (\delta_{\theta_i^+} -  \delta_{\theta_i^-}) + \mu_S(\Theta)\delta_e + \mu^*, \quad \text{ for }\theta_i^+,\theta_i^- \in \Theta \text{ and }c_i \in \R.   
\end{align*}
Choosing $\Theta = \R^2$, $e=0$, and an activation function such that $\sigma(0) =0$, we have that the term $\mu_S(\Theta)\delta_e$ does not contribute to $f_{\mu_{S}}$ and, since $|\mu_{S}(\Theta)|$ is penalized in the $\KRu$ term in \cref{eq:risk_distillation}, a minimizer of \cref{eq:risk_distillation} has to satisfy $\mu_{S}(\Theta)=0$. For this reason, with the choice $p=2$ the objective in the optimization problem \cref{eq:risk_distillation} reduces to
\begin{equation}
\begin{aligned}
    & F(c,w^+,b^+,w^-,b^-) = \frac{1}{N}\sum_{j=1}^N \Bigg|\sum_{i=1}^{M} c_i(\sigma(\langle w_i^+,x^j\rangle +  b_i^+) - \sigma(\langle w_i^-,x^j\rangle + b_i^-)) \\ 
    & \qquad +  \sum_{i=1}^{M_T} a_i\sigma(\langle w_i,x^j\rangle + b_i) - y^j\Bigg|^2 
    + \alpha \left\|\sum_{i=1}^M c_i (\delta_{\theta_i^+} -  \delta_{\theta_i^-})\right\|_{\KR} \\
    & \qquad + \beta \sum_{i=1}^M |c_i|(1+ \|(w^-_i, b^-_i)\|^2 + \|(w^+_i, b^+_i)\|^2).
\end{aligned}
\end{equation}

In order to deal with the $\KR$ norm, one could compute an optimal transport problem at each iteration but here we choose to follow a simpler path. By considering the transport plan that couples each $\theta_i^+$ with the corresponding $\theta_i^-$, we have 
\begin{equation}\label{eq:coupled_plan}
    \left\|\sum_{i=1}^M c_i(\delta_{\theta_i^+} - \delta_{\theta_i^-})\right\|_{\KR}\leq \sum_{i=1}^M |c_i| \|(w^+_i, b^+_i) - (w^-_i, b^-_i)\|_2.
\end{equation}
In the experiment, we let $\theta_i^+$ and $\theta_i^-$ start from the same position and, given the fact that the $\KR$ penalization tends to move them closer, we suppose that $\theta_i^+$ would not go too far with respect to $\theta_i^-$, and vice versa, during training. In this setting, the transport plan that couples each $\theta_i^+$ with $\theta_i^-$ is optimal (if $\theta_i^+$ is far from any $\theta_j^+$ and $\theta_j^-$, with $j\neq i$) and the inequality in \cref{eq:coupled_plan} can be regarded as an equality during training. For this reason, in order to give simple and self-contained experiments, we choose to restrict ourselves to the following simpler objective function
\begin{equation}\label{eq:exp_objective}
\begin{aligned}
    & F(c, w^+, b^+, w^-, b^-) = \sum_{j=1}^N \Bigg|\sum_{i=1}^{M} c_i(\sigma(\langle w^+_i,x^j\rangle + b^+_i) - \sigma(\langle w^-_i,x^j\rangle + b^-_i)) \\ 
    & \qquad + \sum_{i=1}^{M_T}  a_i(\sigma(\langle w_i,x^j\rangle + b_i) - y^j\Bigg|^2 + \alpha \sum_{i=1}^{M} |c_i| \|(w^+_i, b^+_i) - (w^-_i, b^-_i)\|_2 \\ 
    & \qquad + \beta \sum_{i=1}^{M} |c_i|(1+ \|(w^-_i, b^-_i)\|^2 + \|(w^+_i, b^+_i)\|^2).
    \end{aligned}
\end{equation}

\paragraph{Experiment design} We consider $\mu^*$ to be a teacher network that was trained to learn a function $\tilde g:\R\to \R$ (actually, we consider $\tilde g = f_{\mu^*}$ for simplicity, hence the teacher learns the function $\tilde g$ exactly). We suppose now that the physics of the problem has changed and that the true function the student network has to reconstruct is given by a shifted version of $\tilde g$, given by $g(x)=\tilde g(x-\delta)$, with $\delta > 0$. The data $(x_j,y_j)_{j=1}^N$ are therefore generated using $g$ and possibly some added noise. With this, our goal becomes clear: the function we want to reconstruct should be similar to the teacher function, but shifted. In particular, we expect to reconstruct a function of the form \[f_{\mu^*}(x-\delta) = \sum_{i=1}^{M_T}a_i\sigma(w_i (x-\delta)+ b_i)=\sum_{i=1}^{M_T}a_i\sigma(w_i x+ b_i - w_i \delta).\] Therefore, we expect to learn a shift in the parameters $b_i$ similar to $w_i \delta$. In particular, we expect the dipoles to position in a way to remove mass from the position $(w_i,b_i)$ and add it to the position $(w_i, b_i')$ with $b_i'=b_i - w_i \delta$. To tackle this problem, we optimize the function in \cref{eq:exp_objective} using Adam \cite{adam}. Furthermore, we compare the resulting student network with another one minimizing
\[\inf_{\mu \in \M(\R^2)} \frac{1}{N} \sum_{i=1}^N L(y_i, f_\mu(x_i)) + \alpha_{\TV}\|\mu - \mu^*\|_{\TV}.\]
Again, because of the representer theorem in \cite{bartolucci2023understanding}, we can write a minimizer of this problem with a combination of $M_{TV}\leq N$ neurons as
\[\mu_{\TV} = \sum_{i=1}^{M_{\TV}} \bar a_i \delta_{(\bar w_i, \bar b_i)} + \mu^*.\]
The optimization problem reduces therefore to minimizing the following function
\[F(\bar a, \bar w, \bar b) = \sum_{j=1}^N \Bigg|\sum_{i=1}^{M_{\TV}} \bar a_i\sigma(\langle \bar w_i,x^j\rangle + \bar b_i) + \sum_{i=1}^{M_T}  a_i(\sigma(\langle w_i,x^j\rangle + b_i) - y^j\Bigg|^2
    +  \alpha_{TV} \sum_{i=1}^{M_{\TV}} |\bar a_i|.\]
We report the results for two different choices of activation function. The results for $\sigma(x) = \sin(x)$ are shown in \cref{figure:1} and for $\sigma(x) = \tanh(x)$  in \cref{figure:2}. We choose $M = M_{\TV}= M_T = 5$, $N=100$ and initialize the weights as $(w^+,b^+)=(w^-,b^-) = (w,b)$ and $(\bar w, \bar b) =0$. We set $\delta = 0.2$, $\alpha = 0.01$, $\beta = 0.00001$ and $\alpha_{\TV}=0.01$.
\begin{figure}
    \centering
    \includegraphics[width=\linewidth]{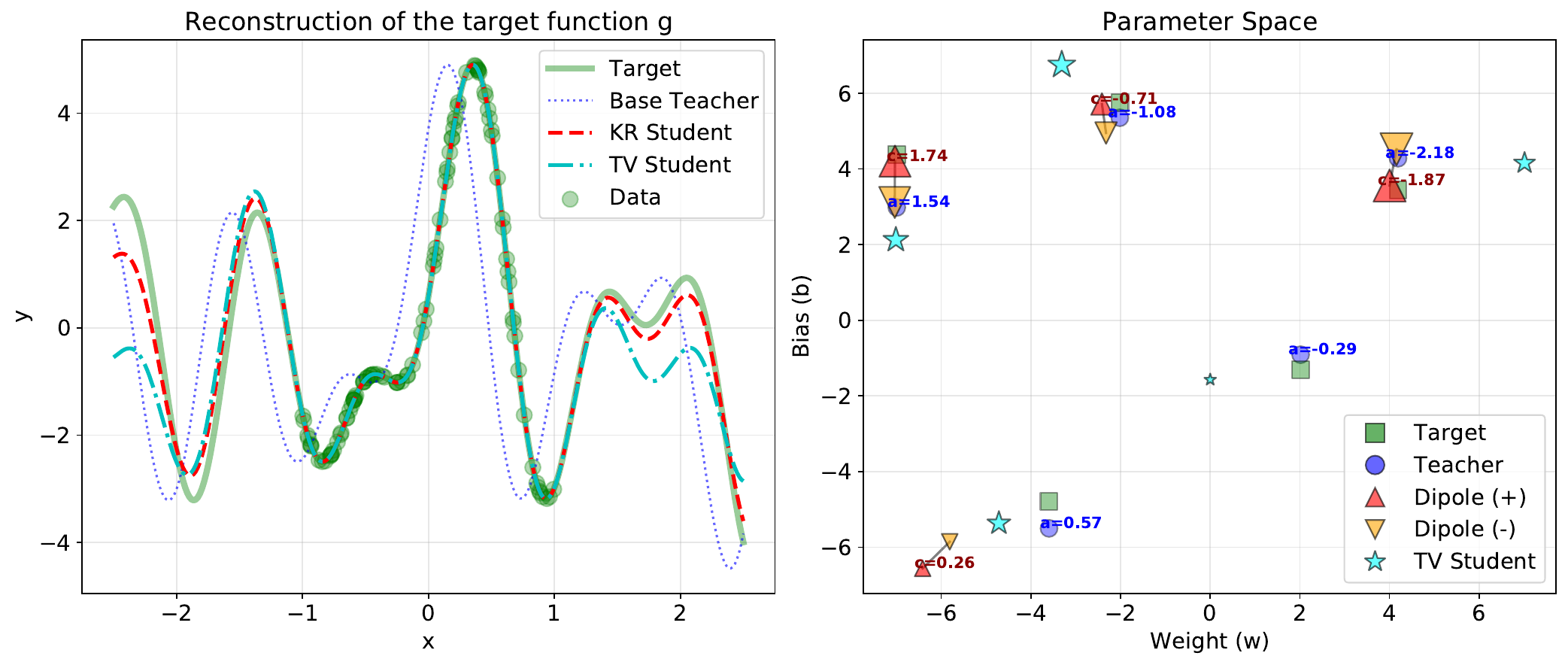}
    \caption{Numerical results with the activation function $\sigma(x) = \sin(x)$. Left: the functions $g$, $f_{\mu^*}$, $f_{\mu_S}$ and $f_{\mu_{\TV}}$. Right: the measures $\mu^*$, the target shifted measure with $b_i'= b_i - w_i \delta$ and the two reconstructions, $\mu_S$ and $\mu_{\TV}$. One can see that the dipoles of the student network that uses the $\KR$ regularization approximately reconstruct the shift, leading to a smaller error on the test set (MSE on the whole interval $[-2.5, 2.5]$). In fact, in this picture we obtain an MSE in the test set of $0.11426929$ for the KR student and of $0.70984238$ for the TV student.}
    \label{figure:1}
\end{figure}

\begin{figure}
    \centering
    \includegraphics[width=\linewidth]{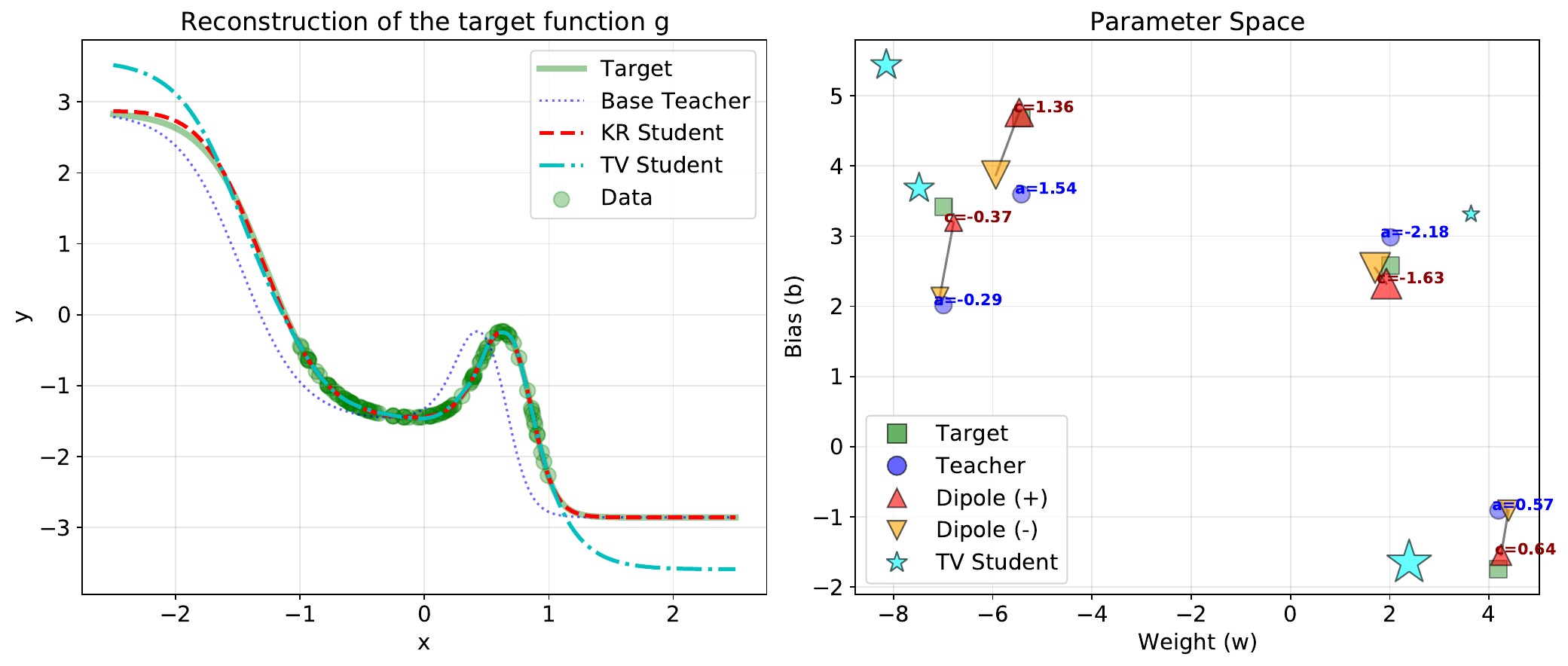}
    \caption{Numerical results with the activation function $\sigma(x) = \tanh(x)$. Left: the functions $g$, $f_{\mu^*}$, $f_{\mu_S}$ and $f_{\mu_{\TV}}$. Right: the measures $\mu^*$, the target shifted measure with $b_i'= b_i - w_i \delta$ and the two reconstructions, $\mu_S$ and $\mu_{\TV}$. One can see that the dipoles of the student network that uses the $\KR$ regularization approximately reconstruct the shift, leading to a significantly smaller error on the test set (MSE on the whole interval $[-2.5, 2.5]$). In fact, in this picture we obtain an MSE in the test set of $0.00119701$ for the KR student and of $0.16099747$ for the TV student.}
    \label{figure:2}
\end{figure}

The code has been implemented in Python on a laptop with Intel Core i7 1165G7 CPU @ 2.80GHz and 8 Gb of RAM. The code is available at \url{https://github.com/EmanueleNaldi/Training-with-KR-norms}.

%%%%%%%%%%%%%%%%%%%%%%%%%%
\section*{Acknowledgements}
FB and YK acknowledge the Royal Society International Exchanges grant IES{\textbackslash}R1{\textbackslash}241209. 
YK also acknowledges the support of the EPSRC (Fellowship EP/V003615/2 and Programme Grant EP/V026259/1) and the National Physical Laboratory. The research by EN has been supported by the MUR Excellence Department Project awarded to the department of mathematics at the University of Genoa, CUP D33C23001110001, and by the US Air Force Office of Scientific Research (FA8655-22-1-7034). FB and EN are members of GNAMPA of the Istituto Nazionale di Alta Matematica (INdAM).
SV was partially supported by the MUR Excellence Department Project MatMod@TOV awarded to the Department of Mathematics, University of Rome Tor Vergata, CUP E83C18000100006.
SV also acknowledges financial support from the MUR 2022 PRIN project GRAFIA, code 202284Z9E4.

%%%%%%%%%%%%%%%%%%%%%%%%%%%%%%%%%%%%%%%%%

\printbibliography

\end{document}